\documentclass[12pt]{amsart}

\usepackage{amsmath,amsfonts,amsthm,amssymb,amscd}

\usepackage{enumerate}

\usepackage{graphicx}

\usepackage[T1]{fontenc}

\usepackage{tikz}
\usepackage{hyperref}
\usepackage[alphabetic]{amsrefs}

\usepackage[hmargin=34mm,vmargin=41.6mm,footskip=5ex]{geometry}

\makeatletter
\@namedef{subjclassname@2010}{%
  \textup{2010} Mathematics Subject Classification}
\makeatother

\theoremstyle{definition}

\theoremstyle{plain}

\newtheorem{twr}{Theorem}[section]
\theoremstyle{definition}
\newtheorem{uwg}[twr]{Remark}
\newtheorem{kon}[twr]{Convention}

\newtheorem{defi}[twr]{Definition}

\theoremstyle{plain}
\newtheorem{lem}[twr]{Lemma}
\newtheorem{stw}[twr]{Proposition}

\numberwithin{equation}{section}

\def\bt{\begin{twr}}
\def\et{\end{twr}}
\def\bs{\begin{stw}}
\def\es{\end{stw}}
\def\bp{\begin{proof}}
\def\ep{\end{proof}}
\def\hi{\angle}
\def\ku{\kappa}
\def\kub{\kappa_{\partial}}
\def\mat{\mathrm}

\begin{document}

\title[Systolic groups are not torsion]{Infinite systolic groups are not torsion}

\author[T. Prytu{\l}a]{Tomasz Prytu{\l}a}
\thanks{The author was supported by the Danish National Research Foundation through the Centre for Symmetry and Deformation (DNRF92)}

\address{School of Mathematics, University of Southampton, Southampton SO17 1BJ, UK}
\email{t.p.prytula@soton.ac.uk}

\date{}

\begin{abstract}
We study $k$--systolic complexes introduced by T.\ Januszkiewicz and J.\ \'{S}wi\k{a}tkowski, which are simply connected simplicial complexes of simplicial nonpositive curvature. Using techniques of filling diagrams we prove that for $k \geq 7$ the $1$--skeleton of a $k$--systolic complex is Gromov hyperbolic. We give an elementary proof of the so-called Projection Lemma, which implies contractibility of $6$--systolic complexes. We also present a new proof of the fact that an infinite group acting geometrically on a $6$--systolic complex is not torsion.
\end{abstract}


\maketitle

\section*{Introduction}\label{r:1}

Simply connected nonpositively curved metric spaces, called $\mathrm{CAT}(0)$ spaces, have been intensively studied for over 50 years, and they are one of the major parts of geometric group theory~\cite{bridson1999metric}. Because of a theorem by Gromov, a special place among $\mathrm{CAT}(0)$ spaces is occupied by $\mathrm{CAT}(0)$ cube complexes: a cube complex is $\mathrm{CAT}(0)$ if and only if it is simply connected and satisfies an easily checked, local combinatorial condition (links are flag).

Therefore naturally arises the question if there exists a similar characterization for simplicial complexes. A partial answer to that question is the notion of \emph{systolic complexes}. These are simply connected simplicial complexes, whose links also satisfy certain combinatorial condition called \emph{$6$--largeness}. This makes them good analogues of $\mathrm{CAT}(0)$ cube complexes. This condition may be treated as an upper bound for simplicial curvature around the vertex, and hence complexes with $k$--large links for $k \geq 6$ are also called complexes of \emph{simplicial nonpositive curvature} (SNPC). The $k$--largeness condition is geometrically motivated by the $2$--dimensional case, where systolic complexes actually are $\mathrm{CAT}(0)$. In general case SNPC does not imply nor is implied by the metric nonpositive curvature.

Systolic complexes were introduced by V.~Chepoi under the name \emph{bridged complexes} in \cite{chepoi2000graphs}, although their $1$--skeleta, the \emph{bridged graphs}, were studied earlier by V.~Chepoi, R.~E.~Jamison, M.~Farber and V. P. Soltan  \cite{soltanchepo, farberjami, chepoialgo}. 

In the context of group theory, systolic complexes were rediscovered independently by F.~Haglund~\cite{haglund2003complexes} and by J.~\'{S}wi\k{a}tkowski and T.~Januszkiewicz~\cite{januszsimplicial}, and they were used to construct word-hyperbolic groups of large cohomological dimension, that do not come from isometries of the hyperbolic $n$--space. However, systolic complexes turned out to be an interesting object of study on their own. It has been shown that in many situations systolic complexes behave like $\mathrm{CAT}(0)$ spaces, and share many of their properties \cite{januszsimplicial, przytycki2007systolic, elsner2009flats, elsner2009isometries}.

In the current work we mainly focus on analogies between systolic complexes and $\mathrm{CAT}(0)$ spaces. We present new and often more direct proofs of various results of \cite{januszsimplicial}, mostly using techniques of \emph{filling diagrams}.
Since our article is self-contained it may be also treated as an introduction to the theory of systolic complexes. Another introductory source is \cite{swiatkowski2006}.

Let us briefly describe the content of the article. In Section~\ref{r:21} we give preliminary definitions and fix the notation. Then we pass to Section~\ref{r:23} where we define $k$--large and $k$--systolic complexes, and we introduce combinatorial tools to study these complexes. In Section~\ref{r:24} we prove that $1$--skeleton of $k$--systolic complex is $\delta$--hyperbolic for $k \geq 7$. Proofs in these sections are usually simplified versions of ones in~\cite{januszsimplicial}, however, we include them for the integrity of the article. From Section~\ref{r:projection} on our research focuses on the case where $k=6$. We give an elementary proof of the simplicial version of Cartan--Hadamard Theorem, which states that systolic complexes are contractible. Finally we turn to group theory. In Section~\ref{r:directed} we discuss properties of \emph{directed geodesics}, and in Section~\ref{r:torsion} we use them to give a direct (i.e. not invoking biautomaticity) proof of a theorem that an infinite group acting geometrically on a systolic complex is not torsion.

\subsection*{Acknowledgements}I would like to thank my advisor Piotr Przytycki for patience, help and guidance, of which he was so full during the year of working on this article. I would also like to thank anonymous Referees for constructive remarks.

\section{Notation, basic definitions and combinatorial Gauss--Bonnet }\label{r:21}

Let $X$ be a simplicial complex. We do not assume that $X$ is finite dimensional nor that it is locally finite. However, when $X$ is finite dimensional, we define its dimension $\mat{dim}X$ to be the largest number $n$, such that $X$ contains an $n$--simplex. We equip $X$ with a CW--complex topology, and always consider $X$ as a topological space, rather than an abstract simplicial complex. Most of the time though we are interested only in the combinatorial structure of $X$.

Given a subset of vertices $\{v_1, \ldots, v_n\}$ of $X$ let $\mat{span} \{v_1, \ldots, v_n\}$ denote the largest subcomplex of $X$ that has $\{v_1, \ldots, v_n\}$ as its vertex set. We will call it a subcomplex \emph{spanned by} $\{v_1, \ldots, v_n\}$. If this subcomplex is a simplex, we denote it by $[v_0,\ldots,v_n]$. A subcomplex spanned by a collection of subcomplexes $\{X_1, \ldots, X_n\}$ is defined as the subcomplex spanned by all the vertices of all $X_i$'s.

We say that $X$ is \emph{flag} if every set of vertices pairwise connected by edges spans a simplex of $X$. 
By a \emph{cycle} in $X$ we understand a subcomplex homeomorphic to $S^1$. If $\gamma$ is a cycle, by $|\gamma|$ we denote the number of edges in $\gamma$ and we call it the \emph{length} of $\gamma$. A \emph{diagonal} in $\gamma$ is an edge connecting two nonconsecutive vertices of $\gamma$. A \emph{link} of a vertex $v$, denoted by $X_v$, is a subcomplex of $X$ that consists of all simplices $\sigma \in X$ that do not contain $v$, but that together with $v$ span a simplex of~$X$.

A \emph{simplicial map} $f\colon X \to Y$ between simplicial complexes $X$ and $Y$ is a map which sends vertices to vertices, and if vertices $v_0,\ldots,v_k \in X$ span a simplex $\sigma$ of $X$ then their images span a simplex $\tau$ of $Y$ and we have $f(\sigma)=\tau$. Therefore a simplicial map is determined by its values on the vertex set of $X$. A simplicial map is called \emph{nondegenerate}, if it is injective on each simplex.

We now introduce the notion of combinatorial curvature and the combinatorial version of the Gauss--Bonnet Theorem.

\begin{defi} Let $S$ be a triangulated surface, possibly with a boundary. For a vertex $v \in S$, by $\hi_S(v)$ we denote the number of triangles ($2$--simplices) of $S$ containing $v$. (If $S$ is clear out of context we may skip the subscript and write simply $\hi(v)$.) Let $\mathrm{int} S$ denote the set of interior vertices of the triangulation of $S$, and $\partial S$ the set of boundary vertices. Then we define:
\begin{itemize}
\item for $v \in \mathrm{int} S$, the \emph{curvature} of $v$ by $\kappa(v)=6-\hi_S(v)$,
\item for $v \in \partial S$, the \emph{boundary curvature} of $v$ by $\kub(v)=3-\hi_S(v)$.
\end{itemize}
\end{defi}

\begin{twr}{\emph{(Combinatorial Gauss--Bonnet)}}  Let $S$ be a compact triangulated surface. Let $\chi(S)$ denote the Euler characteristic of $S$. Then the following formula holds:
\[6\chi(S)= \sum_{v \in \partial S} \kub(v) + \sum_{v \in \mathrm{int} S} \kappa(v).\]

\end{twr}
\begin{proof}
We give a proof only in the case when $S$ is homeomorphic to the $2$--disc. The proof of a general version can be found in~\cite[Theorem~V.3.1]{lyndon1979combinatorial}.

We proceed by induction on the number of triangles in $S$. Since $S$ is contractible, its Euler characteristic is equal to $1$. If $S$ is a single triangle $[v_1,v_2,v_3]$, then each $v_i$ is a boundary vertex and we have $\kub(v_i)=2$, hence $\sum_{i=1}^{3} \kub(v_i)=6=6\chi(S)$.

Now assume that $S$ consists of more than one triangle and choose a triangle $T=[v_1,v_2,v_3]$ in $S$ such that at least one of its edges is a boundary edge. Without loss of generality we may assume that it is $[v_1, v_2]$. We have the following three cases to consider. If vertex $v_3$ is an interior vertex, then we are in situation b) in Figure~\ref{p.bonnet}. If $v_3$ is a boundary vertex, then either both edges $[v_1, v_3]$ and $[v_2,v_3]$ are interior edges or exactly one of them is a boundary edge. In the first case we are in situation c) in Figure~\ref{p.bonnet}, and in the second case we are in situation a).

\begin{figure}[!h]
\centering
\begin{tikzpicture}[scale=1]
\definecolor{lgray}{rgb} {0.850,0.850,0.850}

\draw [fill=lgray] (-3,1) to [out=135, in=0] (-4,1.5) to (-4,1.5) to [out=180, in=90] (-5,0.5) to (-5,0.5) to [out=270, in=180] (-4,-0.5) to (-4,-0.5) to [out=0, in=225] (-3,0) to (-3,1);
\draw (-3,0)--(-3,1)--(-3+0.87,0.5)--(-3,0);
\node [below right] at (-3,0) {$v_3$};
\node [above right] at (-3,1) {$v_1$};
\node [right] at (-3+0.87,0.5) {$v_2$};

\node at (-3+0.3,0.5) {$T$};

\node at (-4,0.5) {$S'$}; 

\node  at (-5.5,1.5) {a)};
\node  at (-0.5,1.5) {b)};

\draw [fill=lgray] (-3+1+4,1) to [out=135, in=0] (-4+1+4,1.5) to (-4+1+4,1.5) to [out=180, in=90] (-5+1+4,0.5) to (-5+1+4,0.5) to [out=270, in=180] (-4+1+4,-0.5) to (-4+1+4,-0.5) to [out=0, in=225] (-3+1+4,0) to (-3+1+4,1);

\draw [fill=white] (-3+1+4,0)--(-3+1+4,1)--(-3+4+1-0.87,0.5)--(-3+4+1,0);

\node [below right] at (-3+1+4,0) {$v_2$};
\node [above right] at (-3+1+4,1) {$v_1$};
\node [ left] at (-3+4+1-0.87,0.5) {$v_3$};

\node  at (-3+4+1-0.3,0.5) {$T$};

\node [above] at (-3+4+1-1,.875) {$S'$};

\node  at (3.75,1.5) {c)};

\draw [fill=lgray] (-3+1+4+5-1.25,1) [out=100, in=60] to (-3+1+4+3+0.8-1.25,1+0.75) [out=240, in=145] to (-3+4+1+5-0.87-1.25,0.5);

\draw [fill=lgray] (-3+1+4+5-1.25,0) [out=260, in=300] to (-3+1+4+3+0.8-1.25,-0.75) [out=120, in=215] to (-3+4+1+5-0.87-1.25,0.5);

\draw [fill=white] (-3+1+4+5-1.25,0)--(-3+1+4+5-1.25,1)--(-3+4+1+5-0.87-1.25,0.5)--(-3+4+1+5-1.25,0);

\node [below right] at (-3+1+4+5-1.25,0) {$v_2$};
\node [above right] at (-3+1+4+5-1.25,1) {$v_1$};
\node [ left] at (-3+4+1-0.87+5-1.25-0.1,0.5) {$v_3$};

\node  at (-3+4+1-0.3+5-1.25,0.5) {$T$};

\node [above] at (-3+4+1-1+5+0.25-1.25,.875+0.05) {$S'_1$};

\node [above] at (-3+4+1-1+5+0.25-1.25,-1.5+.875+0.05) {$S'_2$};

\end{tikzpicture}
\caption{Three possible ways in which triangle $T$ can lie in $S$.} 
\label{p.bonnet}
\end{figure}
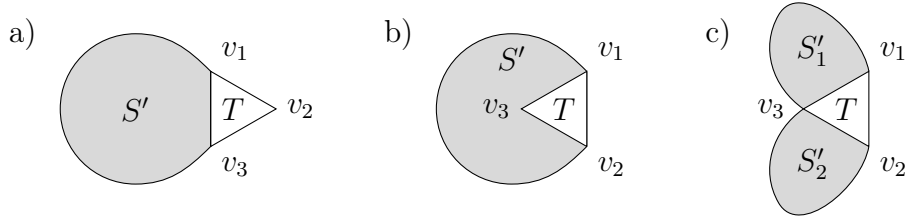

We first consider situations a) and b). Remove $T$ from $S$ and call the resulting surface $S'$. By the inductive assumption $S'$ satisfies the Gauss--Bonnet formula. We will show that after adding $T$ back the formula still holds.
In situation a) after adding $T$, the boundary curvature of both $v_1$ and $v_2$ decreases by $1$. The new vertex $v_3$ contributes to the Gauss--Bonnet sum its boundary curvature which is equal to $2$, hence the whole sum remains unchanged.

In situation b) the vertex $v_3$ becomes an interior vertex, so its curvature increases by $3$, but now it is contained in one more triangle, hence all together its curvature increases by $2$. The boundary curvature of both $v_1$ and $v_2$ again decreases by $1$, so the sum  remains unchanged.

In situation c) we proceed as follows. Consider two subsurfaces $S'_1$ and $S'_2 \cup T$ of $S$ (i.e.\ we cut $S$ along the edge $[v_1,v_3])$. By the assumption both $S'_1$ and $S'_2 \cup T$ have less triangles than $S$, and hence each of them satisfies the Gauss--Bonnet formula.

Let us compare the sum of the Gauss--Bonnet sums of $S'_1$ and $S'_2 \cup T$ with the Gauss--Bonnet sum of $S$. These two sums differ only at summands corresponding to $v_1$ and $v_3$. In the first sum we have four summands \[(3- \hi_{S'_1}(v_1))+ (3- \hi_{S'_2 \cup T}(v_1)) + (3- \hi_{S'_1}(v_3))  +(3- \hi_{S'_2 \cup T}(v_3)).\] 
In the second sum we have two summands \[(3- \hi_{S}(v_1))+ (3- \hi_{S}(v_3)).\] Since clearly \[\hi_{S'_1}(v_i) + \hi_{S'_2 \cup T}(v_i)= \hi_{S}(v_i)\] for $i \in \{1,3\}$, we get that the first sum is bigger than the second sum by $6$. This finishes the proof since by the inductive assumption the first sum is equal to $12$.
\end{proof}

\section{k-largeness condition}\label{r:23}
In this section we introduce two conditions for simplicial complexes called $k$--largeness and local $k$--largeness. The definition is purely combinatorial, but still geometrically inspired by $2$--dimensional case. In terms of these conditions we define $k$--systolic complexes, which are the main objects of our discussion. After defining, we discuss relations between $k$--large and $k$--systolic complexes and at the end we give some examples and non-examples of those.

\begin{defi}
\label{kalargedefi}
Given a natural number $k \geq 4$, a simplicial complex $X$ is $k$--\emph{large} if it is flag and if every cycle of length less than $k$ has a diagonal. We say that $X$ is \emph{locally $k$--large} if links of all vertices of $X$ are $k$--large.
\end{defi}
Observe that if $k\leq m$ then $m$--largeness implies $k$--largeness. Now we will show that $k$--largeness implies local $k$--largeness. Later we will show that under some additional assumptions the converse also holds.

\begin{lem}
\label{lem:globaltolocal}
If a simplicial complex $X$ is $k$--large, then it is locally $k$--large.
\end{lem}

\begin{proof}Let $v_0$ be any vertex of $X$, and let $\gamma$ be a cycle of length $m < k$ in the link $X_{v_0}$. Then $\gamma$ is also a cycle in $X$ so it must have a diagonal, call it $[v_1, v_2]$. Since $X$ is flag, there is a simplex $[v_0, v_1, v_2]\subset X$. This means that the edge $[v_1, v_2]$ belongs to $X_{v_0}$ and thus it is a required diagonal for $\gamma$.
\end{proof}

Here we state the converse to Lemma~\ref{lem:globaltolocal}.

\begin{twr}
\label{twr:localtoglobal}
For $k \geq 6$, let $X$ be simply connected and locally $k$--large simplicial complex. Then $X$ is $k$--large.
\end{twr}

In order to prove this theorem, we need to introduce the notion of a \emph{filling diagram}. Filling diagrams together with the Gauss--Bonnet Theorem will be our fundamental tools in this article. Before introducing filling diagrams we prove the following lemma.

\begin{lem}
\label{lem:shortloop}
Let $X$ be a $k$--large simplicial complex, and let $S_m^1$ denote the triangulation of the circle that consists of $m$ edges, where $m<k$.
Then for any simplicial map $f_0\colon S_{m}^{1} \to X$ there exists a simplicial map $f\colon D \to X$, where $D$ is a triangulated $2$--disc, such that $\partial D = S_{m}^{1}$, $f|_{\partial D}=f_0$ and $D$ has no interior vertices.
\end{lem}
\begin{proof} We proceed by induction on $m$. If $m=3$ then we get the required extension, since $X$ is flag. Assume that $m>3$ and label vertices of the triangulation of $S_{m}^{1}$ by $(v_1,\ldots, v_m)$. Then by the fact that $X$ is $k$--large and that $m<k$, there are two nonconsecutive vertices, say $v_i, v_j \in S_{m}^{1}$ with $i<j$, such that their images under $f_0$ either coincide, or are connected by an edge. Add to $S_{m}^{1}$ an edge $[v_i, v_j]$. Then $S_{m}^{1} \cup [v_i, v_j]$ is the union of two cycles $S_{A}^1=(v_1,\ldots,v_i,v_j,\ldots, v_n)$ and $S_{B}^{1}=(v_i,\ldots,v_j)$. Note that by the choice of $v_i$ and $v_j$, the restrictions of $f_0$ to $S_{A}^{1}$, respectively to $S_{B}^{1}$, denoted by $f_0^A$, respectively $f_0^B$ are well defined simplicial maps. Since $v_i$ and $v_j$ are nonconsecutive, both $S_{A}^{1}$, $S_{B}^{1}$ are shorter than $m$, and hence by inductive hypothesis we have maps $f^A\colon D_{A} \to X$ and $f^B\colon D_{B} \to X$ extending respectively $f_0^A$ and $f_0^B$, such that neither $D_A$ nor $D_B$ has an interior vertex. Gluing $D_{B}$ and $D_{A}$ along the edge $[v_i, v_j]$, and putting $f = f_A \cup f_B$ gives us the required extension of $f_0$ since $D = D_{A}\cup D_{B}$ has no interior vertices.
\end{proof}

\begin{defi}Let $\gamma$ be a cycle in a simplicial complex $X$. A \emph{filling diagram} for $\gamma$ is a simplicial map $f\colon D \to X$, where
$D$ is a triangulated $2$--disc, and $f|_{\partial D}$ maps $\partial D$ isomorphically onto $\gamma$. A filling diagram $f\colon D \to X$ for $\gamma$ is called:
\begin{itemize}
\item
\emph{minimal area} (or \emph{minimal}) if $D$ consists of the least possible number of triangles ($2$--simplices) among filling diagrams for $\gamma$,

\item
\emph{locally $k$--large} if $D$ is a locally $k$--large simplicial complex,

\item
\emph{nondegenerate} if $f$ is a nondegenerate map.

\end{itemize}
\end{defi}

\begin{uwg}We do not assume that $D$ is a simplicial complex, i.e.\ it may have multiple edges and loops.  Consequently the attaching maps of $2$--cells of $D$ can be loops or (not necessarily embedded) cycles of length $2$ or $3$. We will still call this cell structure a ``triangulation''.

A \emph{simplicial map} for a non-simplicial $D$ is defined in the analogous way as in the simplicial case. In particular, a loop is always collapsed to a vertex, a pair of double edges is mapped to a single vertex or edge, and any $2$--cell whose boundary is not an embedded triangle is collapsed to a vertex or an edge.

Observe that a diagram can be locally $k$--large only if $D$ is simplicial, as required in the definition. In fact non-simplicial diagrams appear only in the proof of Theorem~\ref{twr:istniejediagram}.

Finally, note that for a simplicial $2$--disc, being locally $k$--large is equivalent to saying that every interior vertex is contained in at least $k$ triangles.
\end{uwg}


\begin{twr}
\label{twr:istniejediagram}
Let $\gamma$ be a homotopically trivial cycle in a locally $k$--large complex $X$. Then we have the following:
\begin{itemize}
\item[\emph{(1)}] there exists a filling diagram for $\gamma$,

\item[\emph{(2)}] any minimal filling diagram for $\gamma$ is simplicial, locally $k$--large and nondegenerate.
\end{itemize}
\end{twr}

\begin{proof}
(1) Triangulate $S^1$ with $|\gamma|$ edges, and define $f_0\colon S^1 \to X$ as a simplicial isomorphism from $S^1$ to $\gamma$. Since $\gamma$ is homotopically trivial, the map $f_0$ extends to a map $f\colon D \to X$, where $D$ is a $2$--disc. 
Using relative Simplicial Approximation Theorem \cite[Section 3.4]{spanier} we get that $D$ can be given a triangulation which extends the triangulation of $S^1=\partial D$ and we get a simplicial map $f_1\colon D \to X$, which agrees with $f_0$ on $S^1$. This proves the existence.

(2) Let $f\colon D \to X$ be a minimal filling diagram for $\gamma$. We will show that $D$ is simplicial, locally $k$--large and nondegenerate.

First we prove that $D$ is simplicial. We have to show that there are no double edges and no loops in $D$. We proceed by contradiction. First suppose that we have two edges $e$ and $e'$ joining vertices $v_1$ and $v_2$. Then we can remove a subdisc bounded by $e \cup e'$, and glue $e$ to $e'$, which gives us a triangulation of $D$ with fewer triangles, together with a simplicial map induced by $f$. Since $X$ is simplicial, the images of $e$ and $e'$ under $f$ coincide, hence the induced map is well defined. This contradicts the minimality of $D$.

Now assume that we have a loop in $D$. Then there exists a maximal loop, i.e.\ a loop which is not properly contained in the disc bounded by any other loop. Pick any such loop and call it $e$. Let $D_e$ denote the disc bounded by $e$, and let $T$ be the triangle adjacent to $e$ outside $D_e$. Then we have two possibilities: either two other edges of $T$ are embedded, or they are both loops. Both situations are shown in the Figure~\ref{p:26}:

\begin{figure}[!h]
\centering
\begin{tikzpicture}[scale=1]

\draw (-2-1,1) to [out=-30, in=0] (-2-1, -1);
\node [below] at (-2-1,0.2) {$D_e$};
\draw (-2-1,1) to [out=180+30, in=180] (-2-1, -1);
\node [below] at (-3,-1.1) {$T$};
\node at (-3.75, -0.6) {$e$};
\draw (-2-1,1) to [out=-30, in=90] (-1-1, -0.5);
\draw (-2-1,1) to [out=180+30, in=90] (-3-1, -0.5);
\draw (-1-1,-0.5) to [out=270, in=30] (-2-1,-2);
\draw (-3-1,-0.5) to [out=270, in=150] (-2-1,-2);
\node  at (-3.5-1,1) {a)};

\node  at (0.5-1,1) {b)};

\draw (1,1-0.8) to [out=-30, in=0] (1, -0.5-0.8);
\node [below] at (1,-0.3) {$D_e$};
\node at (1, -1.55) {$e$};
\draw (1,1-0.8) to [out=180+30, in=180] (1, -0.5-0.8);
\draw (1,1-0.8) to [out=30, in=90] (2.5, 1-0.8);
\draw (1,1-0.8) to [out=-30, in=270] (2.5, 1-0.8);
\node [below] at (2,-0.7) {$T$};
\draw (1,1-0.8) to [out=180, in=90] (0.1, 0-0.8);
\draw (0.1, 0-0.8) to [out=270,in=180] (1+0.1, -1.5+0.5-1);
\draw (1+0.1,-1.5+0.5-1) to [out=0, in=270] (3, 0);
\draw (3,0) to [out=90, in=0] (2, 1);
\draw (2,1) to [out=180, in=90] (1, 0.2);
\end{tikzpicture}
\caption{Two possibilities for the loop $e$.}
\label{p:26}
\end{figure}
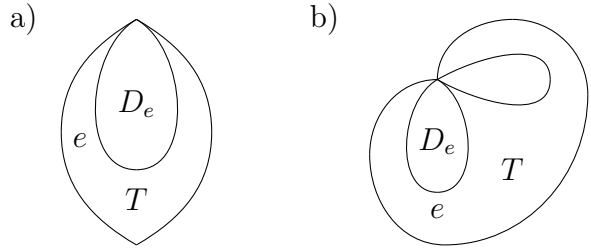

In situation a) two other sides of $T$ have the same endpoints, which is a contradiction with the fact that we do not have double edges.  In situation b) the disc bounded by one of the two other sides of $T$ contains $e$, which contradicts the maximality of $e$.

Now we show that $D$ is locally $k$--large. We do this by showing that there are no interior vertices of $D$ with $\hi(v)<k$.
First observe that in $D$ there are no vertices with $\hi(v) \in \{1,2\}$. Indeed a vertex $v$ with $\hi(v)=1$ would force $D$ to have a loop, and a vertex $v$ with $\hi(v)=2$ would give a loop or two edges with the same endpoints in $D$.

Assume then that we have a vertex $v \in \mat{int} D$ with $3\leq \hi(v) < k$. Then the link $D_v$ is a cycle of length $\hi(v)$. Consider the restriction $f|_{D_v}\colon D_v \to X_{f(v)}$ of $f$. 
It is a simplicial map from a cycle of length $\hi(v) <k$ to the $k$--large complex $X_v$, hence by Lemma~\ref{lem:shortloop} we get a triangulation of a subdisc $D_0$ bounded by $D_v$ with no interior vertices, and a simplicial map $f_0\colon D_0 \to X$, such that $f_0$ agrees with $f$ on $D_v$. Thus we can replace the triangulation of $D_0$ with the one given by Lemma~\ref{lem:shortloop}, and we define a map $f_1$ to be $f_0$ on $D_0$ and $f$ elsewhere. This yields the filling diagram $f_1$ for $\gamma$ with fewer simplices than $D$, which is a contradiction. Hence $D$ is locally $k$--large.

The last part is to prove that $D$ is nondegenerate. Assume the contrary. Then it follows that there is an edge $e$ which is mapped by $f$ to a vertex. Take two triangles containing $e$, delete the interior of their union and glue remaining $4$ edges pairwise. Since $X$ is simplicial, the values of $f$ on the glued edges coincide, hence the map induced by the gluing is well defined. This results in a smaller diagram for $\gamma$.\end{proof}

We are now ready to prove Theorem~\ref{twr:localtoglobal}.

\begin{proof}[Proof of Theorem~\ref{twr:localtoglobal}]
First we show that $X$ satisfies the $k$--largeness condition, then we show flagness of $X$. Let $\gamma$ be a cycle in $X$ with no diagonal. We need to show that $|\gamma| \geq k$. Since $X$ is simply connected, by Theorem~\ref{twr:istniejediagram} there exists a locally $k$--large diagram $f\colon D \to X$ for $\gamma$. Then we have the following:
\begin{itemize}
\item the disc $D$ has at least one interior vertex, because $\gamma$ has no diagonal,
\item interior vertices of $D$ are contained each in at least $k$ triangles, because $D$ is locally $k$--large,
\item boundary vertices are contained each in at least $2$ triangles, that is again because there is no diagonal in $\gamma$.
\end{itemize}
In the notation from the Gauss--Bonnet Theorem, the above conditions give us the following inequalities:
\begin{itemize}
\item[(1)]
for a vertex $v \in \partial D$ we have $\hi(v) \geq 2$, so $\kub(v) \leq 1$, and since $|\partial D|=|\gamma|$, we obtain: 
\[\sum_{v \in \partial D} \kub(v) \leq \sum_{|\gamma|} 1=|\gamma|,\]

\item[(2)]
 for a vertex $v \in \mathrm{int}D$ we have $\hi(v) \geq k$, so $\ku(v) \leq 6-k$, and since $k \geq 6$, all terms $\ku(v)$ are nonpositive. Because $\mathrm{int}D$ contains at least one vertex, we get: 
 \[\sum_{v \in \mathrm{int}D} \ku(v)\leq \sum_{v \in \mathrm{int}D} (6-k)\leq (6-k).\]
\end{itemize}

Putting (1) and (2) together into the Gauss--Bonnet formula gives us:
\[6=6\chi(D)= \sum_{v \in \partial D} \kub(v)+ \sum_{v \in \mathrm{int} D} \ku(v) \leq |\gamma| + (6-k).\]

Hence $|\gamma| \geq  6-(6-k)=k$.

Now we show flagness of $X$. Note that since all links of $X$ are by definition flag, it suffices to show that the $2$--skeleton $X^{(2)}$ is flag. Indeed, assume that we have the $1$--skeleton of an $n$--simplex $[v_1,\ldots,v_n]$ in $X$. Then if $X^{(2)}$ is flag, for any vertex $v_i$ we have the $1$--skeleton of an $(n-1)$--simplex in $X_{v_i}$, which by flagness of the link gives an $n$--simplex $[v_1,\ldots,v_n]$ in $X$.

To show that $X^{(2)}$ is flag, let $(v_1,v_2,v_3)$ be a cycle of length $3$ in $X$. We need to show that this cycle bounds a single $2$--simplex of $X$. The complex $X$ is simply connected, so $(v_1,v_2,v_3)$ is homotopically trivial. Thus by Theorem~\ref{twr:istniejediagram} it has a locally $k$--large filling diagram $f\colon D \to X$. If the disc $D$ is a single triangle then we are done, so let us assume that it consists of at least $2$ triangles.
In this case vertices $v_1, v_2, v_3$ are contained each in at least $2$ triangles, hence $\kub(v_i) \leq 1$ for $i \in \{1,2,3\}$. For a vertex $v$ in the interior of $D$, by $k$---largeness we have $\ku(v) \leq 0$. Hence applying the Gauss--Bonnet Theorem we get:
\[6=6\chi(D)= \kub(v_1)+\kub(v_2)+\kub(v_3)+ \sum_{v \in \mathrm{int} \Delta} \ku(v) \leq 3,\]
a contradiction. Therefore $D$ must consist of precisely $1$ triangle, which means that $[v_1,v_2,v_3] \subset X$.
\end{proof}
After this set-up we introduce \emph{systolic complexes}, the main object of our further discussion.

\begin{defi}Let $k \geq 4$. A simplicial complex $X$ is $k$--\emph{systolic} if it is locally $k$--large, connected and simply connected. For $k=6$ we abbreviate $6$--systolic to \emph{systolic}.
\end{defi}

Note that if $k \geq 6$ then Theorem~\ref{twr:localtoglobal} implies that a $k$--systolic complex is $k$--large, so in particular it is flag. Hence any $k$--systolic complex is determined by its $1$--skeleton. Here we give some basic examples of $k$--large and $k$--systolic complexes:
\begin{itemize}
\item
 If $X$ is a graph, then $X$ is $k$--large if and only if $X$ has no cycle of length less than $k$, and $X$ is $k$--systolic if and only if it is a tree,
\item Equilaterally triangulated Euclidean plane 
is systolic,
\item Equilaterally triangulated hyperbolic plane, 
by triangles with angles $\frac{2\pi}{k}$ is $k$--systolic,
\item For any $k\geq 6$, any triangulation of the sphere $S^2$ is not $k$--large, hence it is not systolic (by Gauss--Bonnet),
\item The locally $6$--large triangulation of the torus:

\begin{figure}[!h]
\centering
\begin{tikzpicture}[scale=0.75]
\definecolor{lgray}{rgb} {0.850,0.850,0.850}



\draw [->, ultra thick] (-1,-2)--(1,-2);

\draw [->, ultra thick] (-1,2)--(1,2);

\draw [->, lgray, ultra thick] (2,0)--(1, 2);

\draw [->, lgray, ultra thick] (-1,-2)--(-2,0);

\draw [->, dashed, gray, ultra thick] (1,-2)--(2, 0);

\draw [->, dashed, gray, ultra thick] (-2,0)--(-1,2);


\draw (0.5,-1)--(0,0)--(-0.5,-1)--(-1,-2);
\draw (0.5,-1)--(0,-2)--(-0.5,-1)--(0.5,-1);

\draw (0.5,1)--(0,0)--(-0.5,1)--(-1,2);
\draw (0.5,1)--(0,2)--(-0.5,1)--(0.5,1);

\draw (0+1,0+2)--(-0.5+1,-1+2)--(-1+1,-2+2);
\draw (0.5+1,-1+2)--(0+1,-2+2)--(-0.5+1,-1+2)--(0.5+1,-1+2);

\draw (-1-1,-2+2)--(0-1,-2+2)--(1-1,-2+2);
\draw (0.5-1,-1+2)--(0-1,-2+2)--(-0.5-1,-1+2)--(0.5-1,-1+2);

\draw (-1+1,2-2)--(0+1,2-2)--(1+1,2-2);
\draw (0+1,0-2)--(-0.5+1,1-2)--(-1+1,2-2);

\draw (0.5+1,1-2)--(0+1,2-2)--(-0.5+1,1-2)--(0.5+1,1-2);

\draw (0.5-1,1-2)--(0-1,2-2)--(-0.5-1,1-2)--(0.5-1,1-2);

\end{tikzpicture}
\caption{The locally $6$--large torus. Links are $6$--cycles.}
\end{figure}
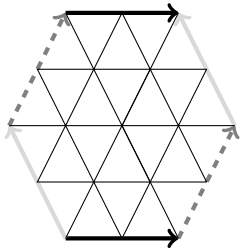
\item
In a similar way for any $g\geq2$, a closed oriented genus $g$ surface admits a locally $6$--large triangulation.
\end{itemize}

\section{7-systolic complexes are Gromov hyperbolic}\label{r:24}
The aim of this section is to prove that if $X$ is a $7$--systolic complex then $X^{(1)}$, the $1$--skeleton of $X$, with the standard metric is Gromov hyperbolic. The standard metric on $X^{(1)}$ is defined as follows. First declare each edge to be isometric to the Euclidean unit interval. Then for any two points $x_1, x_2 \in X^{(1)}$ define the distance between $x_1$ and $x_2$ to be the length of a shortest path between them, and denote it by $|x_1,x_2|$. Similarly, the length of a path $\gamma$ we denote by $|\gamma|$. Although this metric is defined for all points in $X^{(1)}$, most of the time we are interested in distances between vertices, and the distance between any two vertices is just the number of edges in the shortest path between them. Note that the length of a cycle defined in Section~\ref{r:21} is exactly its length in $X^{(1)}$--metric. Also note that any simplicial map is $1$--Lipschitz with respect to this metric.

\begin{defi}Let $(X,d)$ be a metric space and $x_1, x_2, x_3 \in X$.
\begin{itemize}

\item a \emph{geodesic} joining $x_1$ to $x_2$ is an isometry $\gamma: [0,d(x_1,x_2)] \to X$, such that $\gamma(0)=x_1$ and  $\gamma(d(x_1,x_2))=x_2$.

\item a space $(X,d)$ is called \emph{geodesic} if every two points of $X$ can be joined by a geodesic.

\item
 a \emph{geodesic triangle} in $X$ consists of three points and three geodesics joining these points. The images of these geodesics are called the \emph{sides} of this triangle.
\end{itemize}
\end{defi}



\begin{defi}A geodesic metric space $(X,d)$ is called $\delta$--\emph{hyperbolic} (or \emph{Gromov hyperbolic}) if there exists $\delta >0$, such that for every geodesic triangle $[x_1,x_2,x_3]$ in $X$, each side is contained in the union of $\delta$--neighbourhoods of the two other sides.
\end{defi}

\begin{kon}
Note that $X^{(1)}$ is a geodesic metric space, where geodesics are precisely the shortest paths between points.

From now on by \emph{geodesics} we always mean $X^{(1)}$--geodesics (even if we work in $X$, not $X^{(1)}$) and whenever we refer to metric, we mean $X^{(1)}$--metric. Moreover, since geodesics are injective maps, we may identify a geodesic with its image.

Finally, observe that any geodesic $\gamma$ whose endpoints are vertices, is uniquely determined by the sequence of vertices $(\gamma(0), \gamma(1), \ldots, \gamma(n))$.

\end{kon}

To prove the main theorem of this section we need a lemma which gives an estimate on the boundary curvature of a geodesic path. 

\begin{lem}
\label{lem:kubgeodesic}
Let $D$ be a simplicial $2$--disc, and let $\alpha$ be a geodesic in $D$ contained in $\partial D$. Then we have: \[\sum_{v \in \mathrm{int} \alpha} \kub(v) \leq 1,\] where for a path $\alpha$, by $\mathrm{int}\alpha$ we denote all vertices of $\alpha$ except for its endpoints.
\end{lem}

\begin{proof}Since $\alpha$ is a geodesic, every vertex of $\mat{int}\alpha$  is contained in at least two triangles. If not, then part a) of Figure~\ref{p.shortcuts} shows that there exists a shortcut. Moreover, there cannot be two consecutive vertices each contained in two triangles. If that happens, again there is a shorter path between vertices adjacent to these two, which is shown in part b) of Figure~\ref{p.shortcuts}. And whenever there are two vertices, each contained in $2$ triangles, between them there has to be a vertex contained in at least $4$ triangles as in the part c) of Figure~\ref{p.shortcuts}, otherwise we would be able to find another shortcut. Hence the number of vertices contained in $2$ triangles is greater than the number of vertices contained in $4$ triangles at most by $1$. Since all other vertices are contained in $3$ or more triangles, the claim follows.\end{proof}

\begin{figure}[!h]
\centering
\begin{tikzpicture}[scale=1.125]
\draw (0,0) --(1,0);
\draw [ultra thick] (1,0) --(0.5,0.87) -- (0,0);
\node at (-0.25,1) {a)};
\node at (1.75,1) {b)};
\draw (2,0)--(3,0)--(4,0);
\draw [ultra thick] (2,0)--(2.5,0.87)--(3.5,0.87)--(4,0);
\draw [dashed] (2.5,0.87)--(3,0);
\draw  [dashed] (3,0)--(3.5,0.87);
\node at (4.75,1) {c)};
\draw [ultra thick] (5,0-1)--(5.5,0.87-1)--(6.5,0.87-1)--(7,1.74-1)--(8,1.74-1);
\draw (5,-1)--(6,0-1)--(7,0-1)--(7.5,0.87-1)--(8,1.74-1);
\draw [dashed] (5.5,0.87-1)--(6,0-1);
\draw [dashed] (6,0-1)--(6.5,0.87-1);
\draw [dashed] (6.5,0.87-1)--(7,0-1);
\draw [dashed] (6.5,0.87-1)--(7.5,0.87-1);
\draw [dashed] (7.5,0.87-1)--(7,1.74-1);

\draw [ultra thick] (0,-0.5)--(0.5,-0.5);
\node [right] at (0.5,-0.5) {a part of $\alpha$};
\node [right] at (0.5,-1) {a possible shortcut};
\draw (0,-1)--(0.5,-1);
\end{tikzpicture}
\caption{There are shortcuts in a) and b).}
\label{p.shortcuts}
\end{figure}
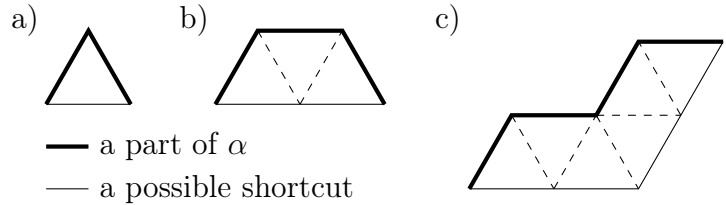


\begin{twr}Let $X$ be a $7$--systolic complex. Then $X^{(1)}$ is Gromov hyperbolic.
\label{twr:gromovhyp}
\end{twr}

\begin{proof}
The proof is due to Januszkiewicz and \'{S}wi\k{a}tkowski \cite[Theorem~2.1]{januszsimplicial}. We need to show that geodesic triangles in $X^{(1)}$ are $\delta$--thin for some $\delta > 0$. First we show that this condition is satisfied by geodesic triangles with vertices in $X^{(0)}$. Note that a geodesic triangle in $X^{(1)}$ can have self-intersections, so it splits into finitely many embedded geodesic bigons and an embedded geodesic triangle, joined by their endpoints or joined by some geodesics between their endpoints.

So to show $\delta$--thinness of any geodesic triangle with vertices in $X^{(0)}$ it is enough to prove that embedded geodesic bigons and embedded geodesic triangles are $\delta$--thin.
Let $\alpha_1 \cup \alpha_2$ be a bigon formed by two geodesics $\alpha_1$ and $\alpha_2$ with common endpoints. Then since $X$ is simply connected, by Theorem~\ref{twr:istniejediagram} there exists a minimal area filling diagram for $\alpha_1 \cup \alpha_2$. Denote this diagram by $f\colon D \to X$. Since $f$ is an isomorphism on $\partial D$ we keep denoting the preimages $f^{-1}(\alpha_i)$ by $\alpha_i$. Now we work in the disc $D$. By Lemma~\ref{lem:kubgeodesic} we have $\sum_{v \in \mathrm{int} \alpha_i} \kub(v) \leq 1$ for both $\alpha_1$ and $\alpha_2$. For $v_0$ and $v_1$, the endpoints of these geodesics, we have $\hi(v_i) \geq 1$ so $\kub(v_i) \leq 2$ for $i\in \{0,1\}$. And finally by Theorem~\ref{twr:istniejediagram} the disc $D$ is locally $7$--large, so for the interior vertices we have $\hi(v) \geq 7$, hence $\ku(v) \leq -1$. Thus by the Gauss--Bonnet formula we get:

\begin{align*}6=6\chi(D)=&~\kub(v_0) +\kub(v_1)+ \sum_{v \in \mathrm{int} \alpha_1} \kub(v) + \sum_{v \in \mathrm{int} \alpha_2} \kub(v) + \sum_{v \in \mathrm{int} D} \ku(v) \leq \\  
\leq&~ 2+2+1+1+ \sum_{v \in \mathrm{int} D} (-1)=6-|\mathrm{int}D|.
\end{align*}

We have $6 \leq 6 -|\mathrm{int}D|$ so $|\mathrm{int}D| \leq 0$. Thus there are no interior vertices in $D$, so every vertex of $\alpha_1$ is connected by an edge to some vertex of $\alpha_2$. This implies that every point from $\alpha_1$ is at distance at most $\frac{3}{2}$ from some vertex of $\alpha_2$ and vice versa, thus $\partial D$ is $\frac{3}{2}$--thin. Then since $f$ is $1$--Lipschitz, it cannot increase distances, hence a bigon $\alpha_1 \cup \alpha_2 \subset X^{(1)}$ is also $\frac{3}{2}$--thin.

Now we pass to geodesic triangles. Let $\alpha_1 \cup \alpha_2 \cup \alpha_3$ be such a triangle. Pick a minimal diagram $f\colon D \to X$ for $\alpha_1 \cup \alpha_2 \cup \alpha_3$. Now $\partial D$ consists of three geodesics: $\alpha_1, \alpha_2$ and $ \alpha_3$. Using the same estimates as above we get:
\begin{align*}6=6\chi(D) =&~ \kub(v_0) +\kub(v_1)+ \kub(v_2)+ \sum_{v \in \mathrm{int} \alpha_1} \kub(v)~+ \\ &+ \sum_{v \in \mathrm{int} \alpha_2} \kub(v) + \sum_{v \in \mathrm{int} \alpha_3} \kub(v) + \sum_{v \in \mathrm{int} D} \ku(v)\leq \\ \leq &~ 2+2+2+1+1+1+ \sum_{v \in \mathrm{int} D} (-1)=9-|\mathrm{int}D|.
\end{align*}
So we have $6 \leq 9  -|\mathrm{int}D|$, hence $|\mathrm{int}D| \leq 3$. There are at most $3$ interior vertices in $D$, thus every point of $\alpha_i$ is at distance at most $4\frac{1}{2}$ from some vertex of $\alpha_j \cup \alpha_k$ where $\{i,j,k\}=\{1,2,3\}$. So $\partial D$ seen as a geodesic triangle is $4\frac{1}{2}$--thin. Since $f$ is $1$--Lipschitz we get that $\alpha_1 \cup \alpha_2 \cup \alpha_3 \subset X^{(1)}$ is also $4\frac{1}{2}$--thin. Thus we proved that geodesic triangles with vertices in $X^{(0)}$ are $4\frac{1}{2}$--thin.

It can be shown that this fact implies that any geodesic triangle in $X^{(1)}$ is $\delta'$--thin, for some other $\delta' >0$. Indeed, given any geodesic triangle $[x_1,x_2,x_3]$ in $X^{(1)}$ one turns it into a geodesic $6$--gon with vertices in $X^{(0)}$, by taking edges containing $x_i$'s to be the new geodesics of length $1$. Then one can show that it is $\delta'$--thin by splitting it into $4$ geodesic triangles with vertices in $X^{(0)}$. Detailed proof can be found in \cite[Theorem~2.1]{januszsimplicial}. \end{proof}


\section{Projection Lemma for systolic complexes}\label{r:projection}
From now our research focuses only on $6$--systolic complexes. Thus we omit $6$ and write \emph{systolic}. An important result which we prove in this section, is the \emph{Projection Lemma}. This theorem, first proved by V. Chepoi~\cite{chepoialgo}
 is an extremely useful tool in working with systolic complexes. Both statement and proof are purely combinatorial: our proof relies on filling diagrams and the Gauss--Bonnet Theorem. Before proving the main theorem we need some preparation.

\begin{defi} Let $X$ be a simplicial complex and let $\sigma$ be a simplex of $X$. We define the \emph{residue} of $\sigma$ in $X$, denoted by $\mathrm{Res}(\sigma, X)$, to be the subcomplex of $X$ spanned by all simplices that contain $\sigma$.
\end{defi}

\begin{defi} Given a simplicial complex $X$ and a vertex $v_0 \in X$ we define the \emph{combinatorial ball} of radius $n$ centered at $v_0$ to be:
\[B_n(v_0, X)=  \mathrm{span}\{v \in X \ | \ |v,v_0| \leq n\},\]
and the \emph{combinatorial sphere}:
\[S_n(v_0, X)= \mathrm{span}\{v \in X \ | \ |v,v_0| = n\}.\]
\end{defi}

The idea of the Projection Lemma is to be able to project simplices contained in $S_n(v_0, x)$ onto $S_{n-1}(v_0, X)$ in a somehow nice way.

\begin{twr}{\emph{(Projection Lemma).}}
\label{twr:projectionlemma}
Let $X$ be a systolic complex and let $v_0$ be a vertex of $X$. Pick $n>0$. Then for any simplex $\sigma \subset S_n(v_0,X)$, the intersection $\mathrm{Res}(\sigma,X) \cap S_{n-1}(v_0, X)$ is a non-empty simplex. We call this simplex the projection of $\sigma$ onto  $S_{n-1}(v_0, X)$ and denote it by $\pi_{v_0}(\sigma)$.\end{twr}

\begin{uwg}
If it is clear in what sphere $\sigma$ is contained, we may also call $\pi_{v_0}(\sigma)$ the projection of $\sigma$ \emph{in the direction of} $v_0$.
\end{uwg}

Note that since $X$ is flag, the intersection $\mathrm{Res}(\sigma,X) \cap S_{n-1}(v_0, X)$ is spanned by the set of vertices in $S_{n-1}(v_0, X)$ which are connected by an edge to every vertex of $\sigma$. So to prove that $\mathrm{Res}(\sigma,X) \cap S_{n-1}(v_0, X)$ is a simplex, it is enough to show that any two vertices in this set are connected by an edge.
The proof is quite long so we divide it into  few steps. First we prove two lemmas about filling diagrams in systolic complexes, then we show that the projection of $\sigma$ is a simplex, and at the end that it is non-empty (this is the hardest part).

\begin{lem}
\label{lem:geodesicbigon1}
Let $\gamma_1 \cup \gamma_2$ be a geodesic bigon between vertices $v_0$ and $v$ in a systolic complex $X$, and let $\gamma_1(0)=\gamma_2(0)=v$. Then $|\gamma_1(1),\gamma_2(1)| \leq 1$, i.e.\ we have one of the situations appearing in Figure~\ref{plemgeodesicbigon1}.
\end{lem}

\begin{proof}If $\gamma_1(1)=\gamma_2(1)$ then we are in situation a) from Figure~\ref{plemgeodesicbigon1}. So assume $\gamma_1(1) \neq \gamma_2(1)$. We will show that in this case we are in situation b). We can assume that $\gamma_1$ and $\gamma_2$ do not intersect except at their endpoints, i.e.\ the bigon $\gamma_1 \cup \gamma_2$ is homeomorphic to $S^1$. If it is not the case, instead of $\gamma_1 \cup \gamma_2$ consider the bigon $\gamma_1|_{[0,m]} \cup \gamma_2|_{[0,m]}$, where $m$ is the smallest number such that $\gamma_1(m)=\gamma_2(m)$.
By the choice of $m$, the new bigon is homeomorphic to $S^1$.

Since $\gamma_1 \cup \gamma_2$ is homeomorphic to $S^1$, as in the proof of Theorem~\ref{twr:gromovhyp}, by Theorem~\ref{twr:istniejediagram} we have a minimal filling diagram $f\colon D \to X$ for $\gamma_1 \cup \gamma_2$. Since the restriction $f|_{\partial D}\colon \partial D \to \gamma_1 \cup \gamma_2$ is an isomorphism, we keep the same notation for the preimages of $\gamma_1 $ and $\gamma_2$ under $f$, and we apply the Gauss--Bonnet Theorem to $D$:
\[6=6\chi(D)=\kub(v_0) + \kub(v) + \sum_{v \in \mathrm{int} \gamma_1} \kub(v)+ \sum_{v \in \mathrm{int} \gamma_2} \kub(v) + \sum_{v \in \mathrm{int} D} \ku(v).\]

\begin{figure}[!h]
\centering
\begin{tikzpicture}[scale=1.25]

\node [above] at (-3,0.5) {a) $|\gamma_1(1),\gamma_2(1)|=0$};
\node [above] at (0,0.5) {b) $|\gamma_1(1),\gamma_2(1)|=1$};

\draw (-3,0)--(-3,-0.87);
\draw [dashed] (-3,-0.87) to [out=315,in=45] (-3,-2.5);
\draw [dashed] (-3,-0.87) to [out=225,in=135] (-3,-2.5);
\node [below] at (-3,-2.5)  {$v_0$};
\node [above] at (-3,0)  {$v$};
\node [left] at (-3,-0.87)  {$\gamma_1(1)$};
\node [right] at (-3,-0.87)  {$=\gamma_2(1)$};


\draw  (-0.5,-0.87)--(0,0)--(0.5, -0.87);
\draw [dashed] (-0.5,-0.87)--(0,-2.5);
\draw [dashed] (0.5,-0.87)--(0,-2.5);
\draw (-0.5,-0.87)--(0.5,-0.87);
\node [below] at (0,-2.5)  {$v_0$};
\node [above] at (0,0)  {$v$};
\node [left] at (-0.5,-0.87)  {$\gamma_1(1)$};
\node [right] at (0.5,-0.87)  {$\gamma_2(1)$};
\end{tikzpicture}
\caption{Distance $|\gamma_1(1), \gamma_2(1)|$ is at most $1$.}
\label{plemgeodesicbigon1}
\end{figure}
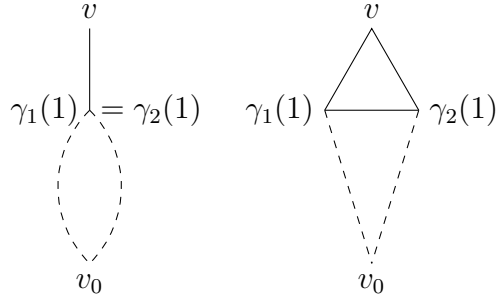

By Theorem~\ref{twr:istniejediagram} disc $D$ is locally $6$--large, so we have $\hi(v) \geq 6$ and thus $\ku(v) \leq 0$ for $v \in \mathrm{int}D$. By Lemma~\ref{lem:kubgeodesic} we get that $\sum_{v \in \mathrm{int} \gamma_i} \kub(v) \leq 1$, hence putting this together gives us:
\begin{align*}
6= 6\chi(D)  \leq&~  1 + 1+ 3-\hi(v_0)  +3- \hi(v)\\
6   \leq&~   8- (\hi(v_0)+ \hi(v))\\
\hi(v_0)+\hi(v)   \leq&~   2.
\end{align*}

Since $\gamma_1 \cup \gamma_2$ is homeomorphic to $S^1$, both $\hi(v_0)$ and $\hi(v)$ are at least $1$. So it must be $\hi(v_0)=\hi(v)=1$, which means that $|\gamma_1(1), \gamma_2(1)|=1$ in $D$. Since $f$ cannot increase this distance, we have $|\gamma_1(1),\gamma_2(1)|=1$ in $X$ as well.\end{proof}
The following lemma is very similar 
to Lemma~\ref{lem:geodesicbigon1}.

\begin{lem}
\label{lem:geodesicbigon2}
Let $v_0$, $v_1$ and $v_2$ be vertices in a systolic complex $X$, such that $v_1, v_2 \in S_n(v_0, X)$ for some $n>0$, and  $|v_1,v_2|=1$. For $i \in \{1,2\}$, let $\gamma_i$ be a geodesic joining $v_0$ with $v_i$ and $\gamma_i(0)=v_i$. Then either $|\gamma_1(1), \gamma_2(1)| \leq 1$, or $|\gamma_1(1), \gamma_2(1)|=2$ with the middle vertex $v$ of a geodesic realizing this distance contained in $S_{n-1}(v_0, X)$, and joined by edges to both $v_1$ and $v_2$ (see Figure~\ref{plemgeodesicbigon2}).\end{lem}

\begin{proof}Consider a cycle $\gamma_1 \cup \gamma_2 \cup [v_1,v_2]$. If $\gamma_1(1)=\gamma_2(1)$ then we are in situation a) in Figure~\ref{plemgeodesicbigon2}. Now assume that $|\gamma_1(1),\gamma_2(1)| \geq 1$. As in the proof of Lemma~\ref{lem:geodesicbigon1} we can assume that cycle $\gamma_1 \cup \gamma_2 \cup [v_1,v_2]$ is homeomorphic to $S^1$. By Theorem~\ref{twr:istniejediagram} we can pick a minimal filling diagram $f\colon D \to X$ for $\gamma_1 \cup \gamma_2 \cup [v_1,v_2]$. Using the same notation and estimates as in the proof of Lemma~\ref{lem:geodesicbigon1}, we get that $\ku(v) \leq 0$ for all $v$ in $\mathrm{int}D$, and $\sum_{v \in \mathrm{int} \gamma_i} \kub(v) \leq 1$ for both $\gamma_1$ and $\gamma_2$. Therefore applying the Gauss--Bonnet Theorem gives the following inequalities:

\begin{align*}
6=\sum_{v \in \gamma_1} \kub(v)+ \sum_{v \in \gamma_2} \kub(v) +  \sum_{v \in \mathrm{int} D} \ku(v) \leq&~1+1+  3-\hi(v_0)~+ \\ &+3-\hi(v_1) +3-\hi(v_2)  \\
6 \leq&~2+ 9 -(\hi(v_0) +\hi(v_1) +\hi(v_2))\\
\hi(v_0) +\hi(v_1) +\hi(v_2) \leq&~5.
\end{align*}

\begin{figure}[!h]
\centering
\begin{tikzpicture}[scale=1.25]

\node [above] at (-2-1,0.5) {a) $|\gamma_1(1),\gamma_2(1)|=0$};

\draw (-2.5-1,0)--(-1.5-1,0)--(-2-1,-0.87)--(-2.5-1,0);
\draw [dashed] (-2-1,-0.87) to [out=315,in=45] (-2-1,-2.5);
\draw [dashed] (-2-1,-0.87) to [out=225,in=135] (-2-1,-2.5);
\node [below] at (-2-1,-2.5)  {$v_0$};
\node [above] at (-2.5-1,0)  {$v_1$};
\node [above] at (-1.5-1,0)  {$v_2$};
\node [left] at (-2-1,-0.87) {$\gamma_1(1)$};
\node [right] at (-2-1,-0.87) {$= \gamma_2(1)$};

\node [above] at (0.5,0.5) {b) $|\gamma_1(1),\gamma_2(1)|=1$};

\draw (0,-0.87)--(0,0)--(1,0)--(1,-0.87);
\draw (0,-0.87)--(1,-0.87);
\draw (0,-0.87)--(1,0);
\draw [dashed] (1,-0.87)--(0,0);
\draw [dashed] (0,-0.87)--(0.5,-2.5);
\draw [dashed] (1,-0.87)--(0.5,-2.5);
\node [below] at (0.5,-2.5)  {$v_0$};
\node [above] at (0,0)  {$v_1$};
\node [above] at (1,0) {$v_2$};
\node [left] at (0,-0.87)  {$\gamma_1(1)$};
\node [right] at (1,-0.87)  {$\gamma_2(1)$};





\node [above] at (6-2,0.5) {c) $|\gamma_1(1),\gamma_2(1)|=2$};

\draw (6.5-3.5, -0.87)--(7-3.5,0)--(8-3.5,0)--(8.5-3.5,-0.87);
\draw (6.5-3.5,-0.87)--(7.5-3.5,-0.87)--(8.5-3.5,-0.87);
\draw (7-3.5,0)--(7.5-3.5,-0.87)--(8-3.5,0);
\draw [dashed] (6.5-3.5,-0.87)--(7.5-3.5,-2.5);
\draw [dashed] (8.5-3.5,-0.87)--(7.5-3.5,-2.5);
\node [below] at (7.5-3.5,-2.5)  {$v_0$};
\node [left] at (6.5-3.5,-0.87)  {$\gamma_1(1)$};
\node [right] at (8.5-3.5,-0.87)  {$\gamma_2(1)$};
\node [below] at (7.5-3.5,-0.87)  {$v$};
\node [above] at (7-3.5,0)  {$v_1$};
\node [above] at (8-3.5,0)  {$v_2$};

\end{tikzpicture}
\caption{Distance $|\gamma_1(1), \gamma_2(1)|$ is at most $2$.}
\label{plemgeodesicbigon2}
\end{figure}
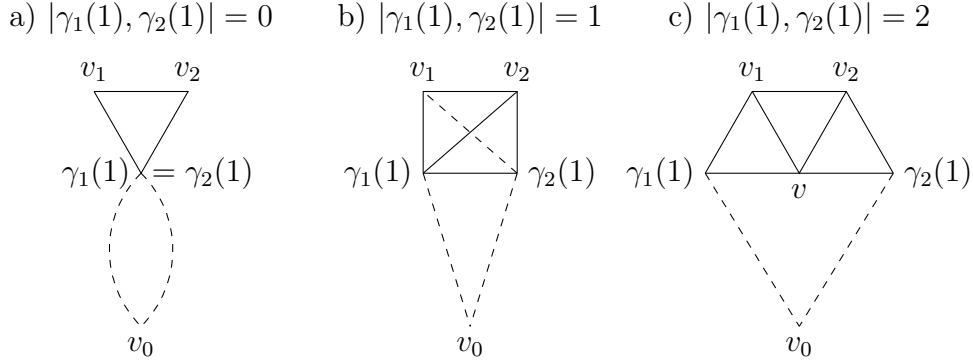
We have $\hi(v_i) \geq 1$ for $i \in \{0,1,2\}$, because $\gamma_1 \cup \gamma_2 \cup [v_1,v_2]$ is embedded, hence solving the last inequality we see that either both $\hi(v_1)$ and $\hi(v_2)$ are equal to $2$, or at least one of them is equal to $1$. Let us consider these two cases.
\begin{itemize}

\item
$\hi(v_1)=1$ (or $\hi(v_2)=1$).\\
This means that $\gamma_1(1)$ is connected by an edge to $v_2$. Then applying Lemma~\ref{lem:geodesicbigon1} to a geodesic bigon formed by $\gamma_2$ and $[v_2, \gamma_1(1)]\cup \gamma_1|_{[1,n]}$ we conclude that $|\gamma_1(1), \gamma_2(1)| = 1$ and we are in situation b) in Figure~\ref{plemgeodesicbigon2}.
\item
$\hi(v_1)=\hi(v_2)=2$.\\
Since both $v_1$ and $v_2$ are contained each in two triangles, we easily see that that there exists a vertex $v$ connected by edges to $v_1, v_2, \gamma_1(1)$ and $\gamma_2(1)$. Thus, it only remains to check that $|v,v_0|=n-1$. Clearly $|\gamma_1(1),v_0|=n-1$ and $|v, \gamma_1(1)|=1$, so $|v_0,v|$ is at most $n$. On the other hand $|v_1, v_0|=n$ and $|v_1,v|=1$, so $|v_0,v| \geq n-1$. Therefore we only need to show, that $|v_0,v|$ cannot be equal to $n$. But if $|v_0,v|=n$ then, again by Lemma~\ref{lem:geodesicbigon1} applied to geodesics $[v,\gamma_1(1)] \cup \gamma_1|_{[1,n]}$ and $[v,\gamma_2(1)] \cup \gamma_2|_{[1,n]}$ we obtain that $\gamma_1(1)$ and $\gamma_2(1)$ are connected by an edge. Hence we are in situation b) in Figure~\ref{plemgeodesicbigon2}. Otherwise $|v,v_0|=n-1$ and we are in situation c) in Figure~\ref{plemgeodesicbigon2}.
\end{itemize}
This proves that for the cycle $\gamma_1 \cup \gamma_2 \cup [v_1,v_2] \subset D$ we have one of the possibilities shown in Figure~\ref{plemgeodesicbigon2}. Since $f$ is simplicial, it cannot increase distances, and therefore we have the same possibilities for $\gamma_1 \cup \gamma_2 \cup [v_1,v_2]$ in $X$.
\end{proof}

Having these two lemmas proved we are ready to prove the Projection Lemma.

\begin{proof}[Proof of Theorem~\ref{twr:projectionlemma}]
First we prove that for a given $\sigma \subset S_n(v_0,X)$, the projection $\pi_{v_0}(\sigma) \subset S_{n-1}(v_0,X)$
is a simplex. Since $X$ is flag, it is enough to show that any two vertices of the projection are connected by an edge. Take two such vertices and call them $w_1$ and $w_2$. They are in $\pi_{v_0}(\sigma)$, so each of them is connected by an edge to every vertex of $\sigma$. Pick any vertex $v \in \sigma$ and consider a geodesic bigon formed by two geodesics: $[v,w_1] \cup \gamma_1$ and $[v,w_2] \cup \gamma_2$, where $\gamma_i$ is a geodesic from $w_i$ to $v_0$. By Lemma~\ref{lem:geodesicbigon1} we have $|w_1,w_2| \leq 1$, so either $w_1=w_2$ or they are connected by an edge. This proves that $\pi_{v_0}(\sigma)$ is a simplex.

Now we need to show that $\pi_{v_0}(\sigma)$ is non-empty. We proceed by induction on the dimension of $\sigma$. First assume $\mat{dim}\sigma=0$, i.e. $\sigma$ is a single vertex $v \in S_n(v_0,X)$. Since $X$ is connected, there is a geodesic $\gamma$ from $v$ to $v_0$, i.e.\ $\gamma(0)=v$. Then $\gamma(1) \in S_{n-1}(v_0,X)$ and $\gamma(1) \in \pi_{v_0}(\sigma)$ by the definition of the projection.

Now assume that $\mat{dim}\sigma=1$ and let $\sigma= [v_1,v_2]$ where $v_1, v_2 \in S_n(v_0,X)$. Choose geodesics $\gamma_1$ and  $\gamma_2$ joining respectively $v_1$ and $v_2$ to $v_0$. Applying Lemma~\ref{lem:geodesicbigon2} to the cycle $[v_1,v_2] \cup \gamma_1 \cup \gamma_2$ we obtain that the projection $\pi_{v_0}([v_1,v_2])$ is non-empty.

For the inductive step let $\mat{dim}\sigma= k-1$ for some $k>2$ and let $\sigma=[v_1,\ldots,v_k]$.  Consider faces $\tau_1 =[v_1,\ldots,v_{k-1}]$ and $\tau_2 =[v_2,\ldots,v_{k}]$. By inductive hypothesis, the projections $\pi_{v_0}(\tau_1)$ and $\pi_{v_0}(\tau_2)$ are non-empty. Choose vertices $w_1 \in \pi_{v_0}(\tau_1)$ and $w_2 \in \pi_{v_0}(\tau_2)$. Since $k >2$, there is a vertex $v_i \in \sigma$ such that $i \notin \{1,k\}$. By construction, vertex $v_i$ is connected by edges to both $w_1$ and $w_2$. Choose geodesics $\gamma_1$ and $\gamma_2$ that join respectively $w_1$ and $w_2$ to $v_0$, and consider the geodesic bigon $[v_i, w_1] \cup \gamma_1  \cup [v_i, w_2] \cup \gamma_2 $, see Figure~\ref{pinductive}. By Lemma~\ref{lem:geodesicbigon1} we have $|w_1,w_2| \leq 1$. 

If $w_1=w_2$ then $w_1$ is connected by edges to all $v_i$ for $i \in \{1, \ldots,k\}$, and hence belongs to the projection  $\pi_{v_0}(\sigma)$.

If $|w_1, w_2|=1$, then we consider the cycle $(w_1, w_2, v_k, v_1)$. Since it has length $4$ it has a diagonal. If this diagonal is $[w_1, v_k]$ then $w_1$ is connected by edges to all $v_i$'s and hence $w_1\in \pi_{v_0}(\sigma)$. If not, then it must be $[w_2 ,v_1]$ and thus $w_2 \in \pi_{v_0}(\sigma)$.
\end{proof}

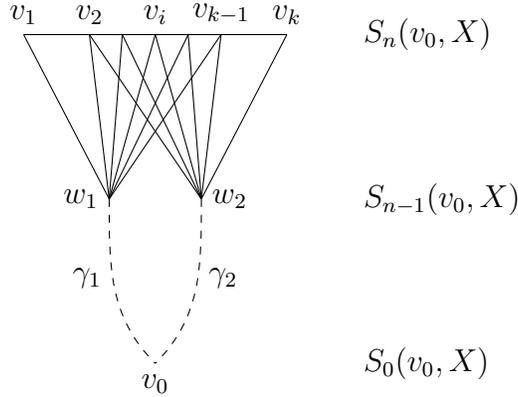
\begin{figure}[!h]
\centering
\begin{tikzpicture}[scale=1.75]
\node [above] at (0,0) {$v_1$};
\node [above] at (0.5,0) {$v_2$};
\node [above] at (1,0) {$v_i$};
\node [above] at (1.5,0) {$v_{k-1}$};
\node [above] at (2,0) {$v_{k}$};

\node [left] at (.65,-1.25) {$w_{1}$};
\node [right] at (1.35,-1.25) {$w_{2}$};
\node [left] at (0.675,-1.85) {$\gamma_{1}$};
\node [right] at (1.325,-1.85) {$\gamma_{2}$};

\draw (0,0)--(0.5,0);
\draw  (0.5,0)--(1.5,0);

\draw (0.75,0)--(0.65,-1.25);
\draw (1.25,0)--(0.65,-1.25);
\draw (1,0)--(0.65,-1.25);
\draw (1.5,0)--(2,0);
\draw (0,0)--(0.65,-1.25);
\draw (0.5,0)--(0.65,-1.25);
\draw (1.5,0)--(0.65,-1.25);

\draw (0.75,0)--(1.35,-1.25);
\draw  (1.25,0)--(1.35,-1.25);
\draw  (1,0)--(1.35,-1.25);
\draw (0.5,0)--(1.35,-1.25);
\draw (1.5,0)--(1.35,-1.25);
\draw (2,0)--(1.35,-1.25);

\draw [dashed] (.65,-1.25) to [out=270, in=135] (1,-2.5);
\draw [dashed] (1.35,-1.25) to [out=270,in=45] (1,-2.5);
\node [below] at (1,-2.5) {$v_{0}$};


\node [right] at (2.5,0) {$S_n(v_0,X)$};
\node [right] at (2.5,-1.25) {$S_{n-1}(v_0,X)$};
\node [right] at (2.5,-2.5) {$S_0(v_0,X)$};

\end{tikzpicture}

\label{pinductive}

\caption{Inductive step.}
\end{figure}

\section{Contractibility of systolic complexes}\label{r:contracto}
Contractibility of systolic complexes follows almost directly from the Projection Lemma. Intuitively, Projection Lemma gives us the possibility to collapse the complex, projecting simplices onto smaller and smaller balls centred at a chosen vertex. In this section we turn this intuition into a proof. Contractibility of systolic complexes has been essentially proved in \cite{anstee}, where the stronger property of \emph{dismantlability} is established for bridged graphs (which are $1$--skeleta of systolic complexes).

\begin{twr}
\label{twr:contraction}
A finite dimensional systolic complex is contractible.
\end{twr}

Note that since locally $6$--large complexes are supposed to be simplicial analogues of locally $\mathrm{CAT}(0)$ spaces, Theorem~\ref{twr:contraction} is a simplicial version of the \emph{Cartan--Hadamard Theorem}. The Cartan--Hadamard Theorem states that the universal cover of a locally $\mathrm{CAT}(0)$ space is 
contractible \cite[Theorem~II.4.1]{bridson1999metric}. In our situation the universal cover of a locally $k$--large complex is systolic, hence contractible. Proof of Theorem~\ref{twr:contraction} is based on two lemmas: the Projection Lemma and the following.

\begin{lem}
\label{lem:maximal}
Let $X$ be a simplicial complex, and let $\sigma \subset X$ be a simplex which 
is properly contained in exactly one maximal (with respect to inclusion) simplex of $X$. Then $X - \mathrm{Res}(\sigma, X)$ is a strong deformation retract of $X$. (Here $X - \mathrm{Res}(\sigma, X)$ denotes the subcomplex of $X$ which is obtained by removing all simplices that contain $\sigma$.)
\end{lem}
\begin{proof}Let $\tau$ be the unique maximal simplex containing $\sigma$.
In the case where $\sigma$ is a codimension $1$ face of $\tau$, the required retraction is called the \emph{elementary collapse} \cite[Section~I.2]{cohen1973course}. To prove general case, we construct the required retraction as the composition of finitely many elementary collapses.

Let $\tau' \subset \tau$ be a codimension $1$ face of $\tau$ such that $\sigma \subset \tau'$. Then consider the set $S=\{\rho \ |\ \sigma \subset \rho\subset \tau'\}$. This is a partially ordered set with respect to inclusion, hence we can extend this order to a linear order and we get $S=\{\sigma= \rho_0 < \ldots < \rho_n = \tau'\}$. Since $\tau' \subset \tau$ is a codimension $1$ face, there is a vertex $v \in \tau$ such that $\tau=\tau' \ast v$. Then we can perform a sequence of elementary collapses for pairs $\rho_i \subset \rho_i \ast v$ starting from $i=n$ with the pair $\tau' \subset \tau$, terminating at $i=0$ with the pair $\sigma=\rho_0 \subset \rho_0 \ast v$. Indeed, at each step $\rho_i$ is a face of exactly one simplex $\rho_i \ast v$, because we already removed all other subsimplices of $\tau$ which might properly contain $\rho_i$.

The composition of these elementary collapses is the required deformation retraction: we removed $\mathrm{Res}(\sigma, X)$ because we removed all the simplices $\rho_i$ and $\rho_i \ast v$ such that $\sigma \subset \rho_i$.
\end{proof}

\begin{proof}[Proof of Theorem~\ref{twr:contraction}]
 Pick any vertex $v_0$ of $X$. By Whitehead's Theorem \cite[Theorem 4.5]{hatcher}, in order to show that $X$ is contractible, it is enough to prove that for any $k \geq 1$, any map $f\colon (S^k, s_0) \to (X, v_0)$ is homotopic to the constant map. The image of $S^k$ is compact, hence it intersects only finitely many simplices of $X$. Thus it is contained in $B_n(v_0, X)$, for $n$ sufficiently large. We prove that for any $n$ the ball $B_n(v_0, X)$ is contractible. We do it by showing how to homotopy retract $B_n(v_0, X)$ onto $B_{n-1}(v_0, X)$, then the result follows by induction since $B_0(v_0, X)=v_0$.

Let $\sigma$ be a maximal simplex in $S_n(v_0, X) \subset B_n(v_0, X)$ and let $m=\mathrm{dim}\sigma$. Then $\sigma \ast \pi_{v_0}(\sigma)$ is a simplex of $B_n(v_0, X)$ containing $\sigma$. We claim that $\sigma \ast \pi_{v_0}(\sigma)$ is the unique maximal simplex containing $\sigma$. To see that, let $\tau$ be a simplex of $B_n(v_0, X)$ with $\sigma \subset \tau$. Let $v$ be a vertex of $\tau$ which does not belong to $\sigma$. Since $\sigma$ is a maximal simplex of $S_n(v_0, X)$, vertex $v$ must lie in $S_{n-1}(v_0, X)$. But $v$ is connected by edges to all vertices of $\sigma$, so by definition it lies in the projection of $\sigma$. Repeating this argument for any such vertex, we finally get that every vertex of $\tau$ either belongs to $\sigma$ or to $\pi_{v_0}(\sigma)$, hence $\tau$ is a subsimplex of $\sigma \ast \pi_{v_0}(\sigma)$. This proves the claim.

Thus $\sigma \subset B_{n}(v_0, X)$ meets the assumptions of Lemma~\ref{lem:maximal}, i.e. it is contained in exactly one maximal simplex of $B_{n}(v_0, X)$. By Lemma~\ref{lem:maximal} the ball $B_{n}(v_0, X)$ homotopy retracts onto $B_{n}(v_0, X) - \mathrm{Res}(\sigma,B_{n}(v_0, X))$. 
We can apply Lemma~\ref{lem:maximal} to all maximal simplices in $S_{n}(v_0, X)$ and obtain a complex which is a deformation retract of $B_{n}(v_0, X)$, where deformation retraction is obtained by performing all retractions arising in Lemma~\ref{lem:maximal} simultaneously. Call this complex $B_{n}(v_0, X)'$, and let $S_{n}(v_0, X)'$ denote $S_{n}(v_0,B_{n}(v_0, X)')$.

By the above procedure we removed all maximal simplices of $S_{n}(v_0,X)$ (removing residue of a simplex removes the simplex itself), thus in $S_{n}(v_0, X)'$ simplices of dimension $m-1$ are maximal. Hence in the same way as above, we show that any such a simplex is contained in a unique maximal simplex of $B_{n}(v_0, X)'$ and we can again apply Lemma~\ref{lem:maximal}.  Applying Lemma~\ref{lem:maximal} to all maximal simplices in $S_{n}(v_0, X)'$ we obtain a complex $B_{n}(v_0, X)''$, with maximal simplices in $S_{n}(v_0, X)''$ having dimension $m-2$. Continuing this procedure, we finally obtain a complex $B_{n}(v_0, X)^{(m+1)}$ with no simplices in $S_{n}(v_0, X)^{(m+1)}$, hence we have $B_{n}(v_0, X)^{(m+1)}=B_{n-1}(v_0, X)$.

Because every map used in the above procedure restricts to the identity on $B_{n-1}(v_0, X)$, we also get that $B_{n-1}(v_0, X)=B_{n}(v_0, X)^{(m+1)}$ is a deformation retract of $B_{n}(v_0, X)$, where deformation retraction is the composition of $m+1$ `simultaneous retractions' arising in the procedure. Now the claim follows by induction. 
\end{proof}

\section{Directed geodesics}\label{r:directed}
In this section we study \emph{directed geodesics}. These are sequences of simplices satisfying certain local conditions, which make them behave better than ordinary geodesics in many situations. Originally they were introduced in \cite[Definition~9.1]{januszsimplicial}. In our case the main purpose of discussing them is to prove Theorem~\ref{nottorsion}. The properties of directed geodesics which we show in this section, namely Theorem~\ref{twr:directedisnormal} and Lemma~\ref{lem:projectionray}, were proved in \cite[Fact~8.3.2, Lemma~9.3, Lemma~9.6]{januszsimplicial}. However, our proofs are different and in some way more elementary.

\begin{defi}
\label{def:directed} A sequence of simplices $(\sigma_0,\ldots,\sigma_n)$ in a systolic complex $X$ is called a \emph{directed geodesic} if it satisfies the following two conditions:
\begin{itemize}

\item[\emph{(i)}]
for any $0\leq i \leq n-1 $ simplices $\sigma_i$ and $\sigma_{i+1}$ are disjoint and together span a simplex of $X$.

\item[\emph{(ii)}]
we have $\mathrm{Res}(\sigma_{i+2}, X) \cap B_1(\sigma_{i}, X)= \sigma_{i+1}$ for any $ 0 \leq i \leq n-2$.

\end{itemize}
Here for a simplex $\sigma$ we define the ball $B_1(\sigma, X)$ to be the simplicial span of the set $\{v \in X \ | \ |v,\sigma| \leq 1\}$.
\end{defi}

\begin{uwg}Above definition is slightly different from the one in \cite[Definition~9.1]{januszsimplicial}: in our case the direction is reversed, i.e.\ originally condition (ii) is stated as $\mathrm{Res}(\sigma_i, X) \cap B_1(\sigma_{i+2}, X)= \sigma_{i+1}$. Our modification is needed to avoid notation problems in the construction of infinite directed geodesics.
\end{uwg}

The following theorem justifies term `geodesic' arising in Definition~\ref{def:directed} and is the crucial step in the proof of Theorem~\ref{nottorsion}.

\begin{twr}
\label{twr:directedisnormal}
Let $(\sigma_0,\ldots,\sigma_n)$ be a directed geodesic in a systolic complex $X$. Then any sequence of vertices $(v_0,\ldots,v_n)$ such that $v_i \in \sigma_i$, is a geodesic in $X^{(1)}$.
\end{twr}
\begin{proof}First we will show that this property is satisfied locally, i.e.\ that any triple $(v_i, v_{i+1}, v_{i+2})$ is a geodesic. Indeed, assume that
$(v_i, v_{i+1}, v_{i+2})$ is not a geodesic, i.e. $|v_i, v_{i+2}|=1$. This means that $v_{i+2} \in B_1(v_i, X) \subset B_1(\sigma_i, X)$ and since $ v_{i+2} \in \mathrm{Res}(\sigma_{i+2}, X)$, we get that $v_{i+2} \in B_1(\sigma_i, X) \cap \mathrm{Res}(\sigma_{i+2}, X)$.  By definition of a directed geodesic $ B_1(\sigma_i, X) \cap \mathrm{Res}(\sigma_{i+2}, X)=\sigma_{i+1}$, hence we get  $v_{i+2} \in \sigma_{i+1 }$, which contradicts the fact that $\sigma_{i+1}$ and $\sigma_{i+2}$ are disjoint. Thus $(v_i, v_{i+1}, v_{i+2})$ is a geodesic. This fact will be used many times later on.

Now observe, that in order to prove the theorem it is enough to prove that $|v_0,v_n|=n$. Indeed, then any sequence of vertices $(v_0, v_1,\dots,v_n)$, where $v_i \in \sigma_i$ gives a path from $v_0$ to $v_n$ of length $n$, hence a geodesic path.

Also note that if a geodesic $\gamma$ from $v_0$ to $v_n$ satisfies $\gamma(i) \in \sigma_i$ for some $1\leq i \leq n-1$, then we are done by induction: assume that the assertion of the theorem holds for any shorter sequence of vertices, then since $\gamma(i) \in \sigma_i$, by inductive hypothesis we have $|v_0,\gamma(i)|=i$ and $|\gamma(i),v_n|=n-i$, hence we get $|v_0, v_n|=n$.

Thus it suffices to show that there is no geodesic from $v_0$ to $v_n$, that is disjoint from the sequence $\sigma_1,\ldots,\sigma_{n-1}$. Assume conversely, that we have such a geodesic and call it $\gamma_0$.

Consider the set of cycles $C=\{(v_0,v_1,\ldots,v_{n-1},v_n) \cup \gamma\}$ where $v_i \in \sigma_i$ and $\gamma$ is a geodesic from $v_0$ to $v_n$ disjoint from $\sigma_i$ for $1\leq i\leq n-1$. By our assumption $C$ is non-empty. Any cycle in $C$ is homotopically trivial, hence by Theorem~\ref{twr:istniejediagram} it has a locally $6$--large filling diagram. Thus we can pick a cycle $(v_0,v_1,\ldots,v_{n-1},v_n) \cup \gamma$ such that its filling diagram $f\colon D \to X$ is minimal among all diagrams for cycles in $C$.


We would like to apply the Gauss--Bonnet Theorem to the diagram $D$, so we need to find certain curvature estimates. We keep the notation $(v_0,v_1,\ldots,v_{n-1},v_n) \cup \gamma$ for the boundary of $D$.

\begin{itemize}

\item
$\sum_{w \in \mathrm{int} \gamma} \kub(w) \leq 0$.\\
 We show that there are no vertices in $\mathrm{int} \gamma$ with positive boundary
curvature. Since $\gamma$ is a geodesic, for any interior vertex $w$ we have $\kub(w) \leq 1$ (cf. Lemma~\ref{lem:kubgeodesic}). 
Suppose there is a vertex $w_i \in \mathrm{int} \gamma$ with $\kub(w_i) = 1$. This means that $w_i$ is contained in two
triangles, call them $[w_{i-1}, w_i, w]$ and $[w_{i+1}, w_i, w]$. Remove these two triangles from $D$ and replace edges $[w_{i-1}, w_i]$ and
$[w_{i}, w_{i+1}]$ by $[w_{i-1}, w]$ and $[w, w_i]$. This gives a new geodesic $\gamma'$ and a new diagram $D'$ with smaller area, which contradicts the minimality of $D$. Figure~\ref{replacement} shows the replacement procedure.

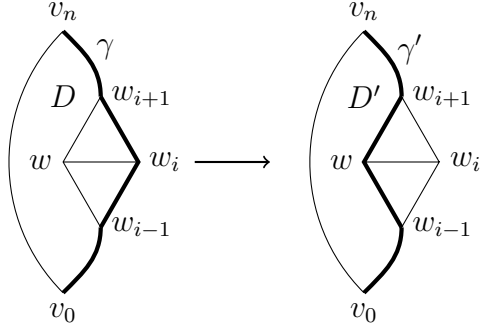
\begin{figure}[!h]
\centering

\begin{tikzpicture}[scale=1]

\draw (-3,0)--(-2,0)--(-2.5,-0.87)--(-3,0)--(-2.5,0.87)--(-2,0);
\draw [ultra thick] (-2.5,0.87) to [out=90,in=315] (-3,1.74);
\draw [ultra thick] (-3,-1.74) to [out=45,in=270] (-2.5,-0.87);
\draw (-3,1.74) to [out=225,in=135] (-3,-1.74);
\draw [ultra thick] (-2.5,-0.87)--(-2,0)--(-2.5,0.87);

\node [below] at (-3,-1.74) {$v_0$};
\node [above] at (-3,1.74) {$v_n$};
\node [left] at (-3,0) {$w$};
\node [right] at (-2,0) {$w_i$};
\node [right] at (-2.5,-0.87) {$w_{i-1}$};
\node [right] at (-2.5,0.87) {$w_{i+1}$};
\node [right] at (-2.70,1.5) {$\gamma$};
\node at (-3,0.87) {$D$};


\draw [->, thick] (-1.25,0)--(-0.25,0);

\draw (-3+4,0)--(-2+4,0)--(-2.5+4,-0.87)--(-3+4,0)--(-2.5+4,0.87)--(-2+4,0);
\draw [ultra thick] (-2.5+4,0.87) to [out=90,in=315] (-3+4,1.74);
\draw [ultra thick] (-3+4,-1.74) to [out=45,in=270] (-2.5+4,-0.87);
\draw (-3+4,1.74) to [out=225,in=135] (-3+4,-1.74);
\draw [ultra thick] (-2.5+4,-0.87)--(-3+4,0)--(-2.5+4,0.87);

\node [below] at (-3+4,-1.74) {$v_0$};
\node [above] at (-3+4,1.74) {$v_n$};
\node [left] at (-3+4,0) {$w$};
\node [right] at (-2+4,0) {$w_i$};
\node [right] at (-2.5+4,-0.87) {$w_{i-1}$};
\node [right] at (-2.5+4,0.87) {$w_{i+1}$};
\node [right] at (-2.7+4,1.5) {$\gamma'$};
\node at (-3+4,0.87) {$D'$};
\end{tikzpicture}
\caption{Replacement procedure.}
\label{replacement}
\end{figure}

\item $\sum_{i=1}^{n-1} \kub(v_i) \leq 1$.\\
First note that there cannot be any vertex $v_i$, with $\kub(v_i)=2$. In this case $v_{i-1}$ is connected by an edge to
$v_{i+1}$ which contradicts the fact that $(v_{i-1}, v_{i},v_{i+1})$ is a geodesic. We will now show that there cannot be two
consecutive vertices, each with boundary curvature equal to $1$. Assume conversely, that we do have two vertices $v_i, v_{i+1}$ each contained in two triangles (one triangle is common for $v_i$ and $v_{i+1}$). Call these triangles $[v_{i-1}, v_i, x]$, $[v_{i}, v_{i+1}, x]$ and $[v_{i+1}, v_{i+2}, x]$. We have $x \not \in \sigma_i$ for otherwise $(x,v_{i+1}, v_{i+2})$ would not be geodesic since $x$ and $v_{i+2}$ are connected by an edge. This, together with the definition of a directed geodesic and the fact that $x \in B_1(v_{i-1}, X)$ implies that $x$ does not belong to $\mathrm{Res}(\sigma_{i+1}, X)$. Hence there exists a vertex $w_{i+1} \in \sigma_{i+1}$ which is not connected by an edge to $x$. Since $w_{i+1}$ belongs to $\sigma_{i+1}$, by definition of a directed geodesic it is connected to all vertices of both $\sigma_i$ and $\sigma_{i+2}$, so in particular it is connected by an edge to $v_{i+2}$ and to $v_i$. Hence we have a cycle $(v_i, x, v_{i+2}, w_{i+1})$ of length $4$. By Theorem~\ref{twr:localtoglobal} the complex $X$ is $6$--large so this cycle has a diagonal (see Figure~\ref{pdirecteddiagonal}). By the choice of $w_{i+1}$ it cannot be $[w_{i+1}, x]$ so it has to be $[v_i, v_{i+2}]$, which contradicts the fact that $(v_i, v_{i+1}, v_{i+2})$ is a geodesic. Hence, there cannot be two consecutive vertices each with boundary curvature equal to $1$.

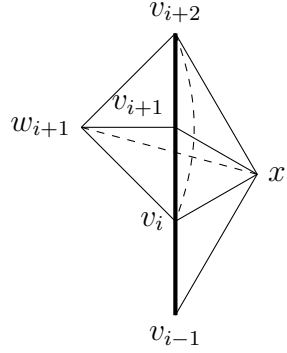
\begin{figure}[!h]
\centering
\begin{tikzpicture}[scale=1.25]
\draw [ultra thick] (0,-2)--(0,-1)--(0,0)--(0,1);
\draw (0,-1)--(0.87,-0.5)--(0,1);
\draw (0,-2)--(0.87,-0.5)--(0,0);
\draw (-1,0)--(0,0);
\draw (-1,0)--(0,1);
\draw (-1,0)--(0,-1);
\draw [dashed] (-1,0)--(0.87,-0.5);
\draw [dashed] (0,1) to [out=290,in=70] (0,-1);
\node [above left] at (0,0) {$v_{i+1}$};
\node [left] at (0,-1) {$v_{i}$};
\node [above] at (0,1) {$v_{i+2}$};
\node [below] at (0,-2) {$v_{i-1}$};
\node [right] at (0.87,-0.5,0) {$x$};
\node [left] at (-1,0) {$w_{i+1}$};
\end{tikzpicture}
\caption{Possible diagonals of the cycle $(v_i, x, v_{i+2}, w_{i+1})$.}
\label{pdirecteddiagonal}
\end{figure}

Next, take any vertex $v_i$ such that $\kub(v_i) =1$ and suppose that $\kub(v_{i+1})=0$. We then have four triangles: $[v_{i-1},
x, v_i]$, $[v_{i}, x, v_{i+1}]$, $[x, y, v_{i+1}]$ and $[v_{i+1}, y, v_{i+2}]$. Again, the vertex $x \not \in
\sigma_i$, so there is a vertex $w_{i+1}$ of $\sigma_{i+1}$ which is not connected by an edge to $x$, but is connected to both $v_i$ and $v_{i+2}$. Thus now we have a cycle $(v_i, x, y, v_{i+2}, w_{i+1})$ of length $5$, which must have a diagonal (see Figure~\ref{pdirecteddiagonal2}). It can be neither $[w_{i+1}, x]$ nor $[v_{i}, v_{i+2}]$. Also it cannot be $[x, v_{i+2}]$ because in that case we would have a cycle of length $4$ and two first edges are the only possible diagonals for that cycle. So the only possibility is that we have diagonals $[v_i, y]$ and $[w_{i+1}, y]$. Then we can do the following: replace the vertex $v_{i+1}$ by $w_{i+1}$ and triangles $[v_{i}, x, v_{i+1}]$, $[x, y, v_{i+1}]$ and $[v_{i+1}, y, v_{i+2}]$ by triangles $[v_i, x, y]$, $[v_i, y, w_{i+1}]$ and $[w_{i+1}, y, v_{i+2}]$, and call the resulting diagram $D'$. Diagram $D'$ is still minimal, but now we have $\kub(v_i) =0$ and $\kub(w_{i+1}) =1$.

\begin{figure}[!h]
\centering
\begin{tikzpicture}[scale=1.25]

\draw [ultra thick] (0,-2)--(0,-1)--(0,0)--(0,1);
\draw (0,-1)--(0.87, -0.5)--(0,0)--(0.87,0.5)--(0,1);
\draw (0.87,-0.5)--(0.87,0.5);
\draw (-1,0)--(0,0);
\draw (-1,0)--(0,-1);
\draw (-1,0)--(0,1);
\draw (-1,0)--(0,0);
\draw (0,-2)--(0.87,-0.5);

\node [below left] at (0,0) {$v_{i+1}$};
\node [left] at (0,-1) {$v_{i}$};
\node [above] at (0,1) {$v_{i+2}$};
\node [below] at (0,-2) {$v_{i-1}$};
\node [left] at (-1,0) {$w_{i+1}$};
\node [right] at (0.87,0.5) {$y$};
\node [right] at (0.87,-0.5) {$x$};

\draw [dashed, thick] (0,-1)--(0.87,0.5);
\draw [dashed, thick] (-1,0)--(0.87,0.5);


\draw [->, thick] (-1.25+2.75,-0.5)--(-0.25+2.75,-0.5);

\draw [ultra thick] (0+4,-2)--(0+4,-1)--(-1+4,0)--(0+4,1);
\draw (0+4,-1)--(0.87+4, -0.5);
\draw (0.87+4,0.5)--(0+4,1);
\draw (0.87+4,-0.5)--(0.87+4,0.5);
\draw (-1+4,0)--(0+4,-1);
\draw (-1+4,0)--(0+4,1);
\draw (0+4,-2)--(0.87+4,-0.5);

\draw (0+4,-1)--(0.87+4,0.5);
\draw (-1+4,0)--(0.87+4,0.5);

\node [left] at (3,0) {$w_{i+1}$};
\node [below left] at (0+4,-1) {$v_{i}$};
\node [above] at (0+4,1) {$v_{i+2}$};
\node [below] at (0+4,-2) {$v_{i-1}$};
\node [right] at (0.87+4,0.5) {$y$};
\node [right] at (0.87+4,-0.5) {$x$};

\end{tikzpicture}
\caption{Pushing upstairs the positive boundary curvature.}
\label{pdirecteddiagonal2}
\end{figure}
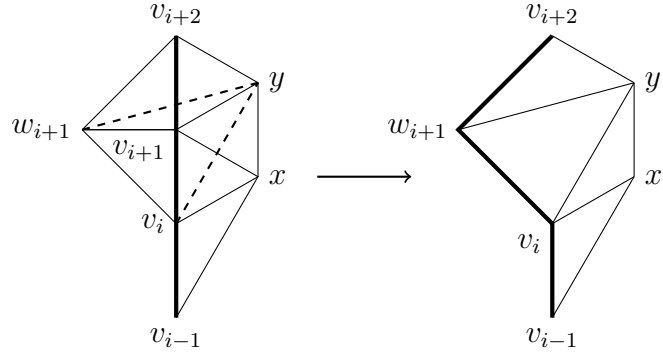

So whenever we have a vertex with boundary curvature equal to $1$ we can `push it upstairs' along vertices with curvature equal to $0$, till we arrive at a vertex with negative curvature or at $v_n$. Hence we have $\sum_{i=1}^{n-1} \kub(v_i) \leq 1$.
\end{itemize}

\begin{itemize}
\item $\sum_{v \in \mathrm{int} D} \kub(v) \leq 0$.\\
Since $X$ is locally $6$--large, for any $v \in \mathrm{int} D$ we have $\kub(v) \leq 0$.

\item $\kub(v_0) \leq 2$, $\kub(v_n) \leq 2$. \\
This follows from the fact that $(v_0,\ldots,v_n) \cup \gamma$ belongs to $C$.
\end{itemize}

To finish the argument we put all these inequalities into the Gauss--Bonnet formula:
 \begin{align*}
 6 &~=\kub(v_0) + \kub(v_n) + \sum_{v \in \mathrm{int} D} \kub(v) + \sum_{i=1}^{n-1} \kub(v_i) +\sum_{v \in \mathrm{int} \gamma} \kub(v) \leq \\
 &~\leq 2+2+0+1+0=5.
 \end{align*}
We get a contradiction, therefore the set $C$ is empty and hence $(v_0,\ldots,v_n)$ is a geodesic.
\end{proof}

Note that the second condition in the definition of a directed geodesic is somehow similar to the condition appearing in the Projection Lemma. Using Theorem~\ref{twr:directedisnormal} we can make this relation explicit.

\begin{lem}
\label{lem:projectionray}
Let $(v_0 = \sigma_0 ,\ldots, \sigma_n)$ be a a sequence of simplices in a systolic complex $X$, starting at a vertex $v_0$. Then $(v_0 = \sigma_0 ,\ldots, \sigma_n)$ is a directed geodesic if and only
if for all  $i \in \{0,\ldots n-1\}$ the simplex $\sigma_i \subset S_i(v_0,X)$ is the projection of $\sigma_{i+1}$ in the direction of $v_0$.
\end{lem}

\begin{proof}
First assume that $(v_0 ,\ldots, \sigma_n)$ is a directed geodesic. Let us recall the definition of the projection of $\sigma_i$ in the direction $v_0$. Note that by Lemma~\ref{twr:directedisnormal} for any $i$ we have $\sigma_i \subset B_i(v_0, X)$, so in particular the projection makes sense:
\[\pi_{v_0}(\sigma_{i+2})=\mathrm{Res}(\sigma_{i+2}, X) \cap B_{i+1}(v_0, X).\]
And by the definition of a directed geodesic we have:
\[\sigma_{i+1}=\mathrm{Res}(\sigma_{i+2}, X) \cap B_1(\sigma_{i}, X).\]
We have $\sigma_i \subset S_i(v_0, X)$ so $B_1(\sigma_{i}, X) \subset B_{i+1}(v_0, X)$, hence $\sigma_{i+1} \subset \pi_{v_0}(\sigma_{i+2})$. We need to show the other inclusion. Assume inductively that for all $0 \leq k \leq i$ we have $\sigma_{k} = \pi_{v_0}(\sigma_{k+1})$. The projection of a simplex is contained in the projection of any of its faces (follows directly from the definition), so since $\sigma_{i+1} \subset \pi_{v_0}(\sigma_{i+2})$, we get that $\pi_{v_0}(\pi_{v_0}(\sigma_{i+2})) \subset \pi_{v_0}(\sigma_{i+1})$ and the latter equals $\sigma_i$ by the inductive assumption. Hence $\pi_{v_0}(\sigma_{i+2}) \subset B_1({\sigma_i}, X)$ and therefore $\pi_{v_0}(\sigma_{i+2}) \subset \sigma_{i+1}$ by the definition of a directed geodesic. Thus we have $\pi_{v_0}(\sigma_{i+2}) = \sigma_{i+1}$. Clearly for $i=0$ we have $v_0 = \pi_{v_0}(\sigma_{1})$ so the claim follows by induction.

To prove the other direction assume $\sigma_k= \pi_{v_0}(\sigma_{k+1})$ for all $0\leq k \leq n$. Condition (i) of Definition~\ref{def:directed} is satisfied by the definition of the projection of a simplex, thus we need to prove that condition (ii) holds. We have $\sigma_{i}=\mathrm{Res}(\sigma_{i+1}, X) \cap B_i(v_0, X)$, and we need to show that this is equal to $\mathrm{Res}(\sigma_{i+1}, X) \cap B_1(\sigma_{i-1}, X)$.

By the assumption $\sigma_{i-1} \subset S_{i-1}(v_0, X)$ so $B_1(\sigma_{i-1}, X) \subset B_i(v_0, X)$, and hence we get that $\mathrm{Res}(\sigma_{i+1}, X) \cap B_1(\sigma_{i-1}, X) \subset \mathrm{Res}(\sigma_{i+1}, X) \cap B_i(v_0, X)=\sigma_i$. Conversely, we have $\sigma_{i-1}= \pi_{v_0}(\sigma_i)$ so $\sigma_{i} \subset B_1(\sigma_{i-1},X)$, and thus we get the other inclusion $\sigma_i = \mathrm{Res}(\sigma_{i+1}, X) \cap B_i(v_0, X)  \subset \mathrm{Res}(\sigma_{i+1}, X) \cap B_1(\sigma_{i-1}, X)$. So $\mathrm{Res}(\sigma_{i+1}, X) \cap B_1(\sigma_{i-1}, X) = \sigma_i$, and therefore the triple $(\sigma_{i-1}, \sigma_i, \sigma_{i+1})$ satisfies condition (ii) of Definition~\ref{def:directed}. This argument works for any $0 \leq i \leq n-2$, hence the claim follows.\end{proof}

\begin{uwg}There is a more general version of the Projection Lemma, which allows one to project in the direction of an arbitrary simplex (see \cite[Lemma~7.7]{januszsimplicial}). Lemma~\ref{lem:projectionray} is then also true if we let $\sigma_0$ be any simplex, not necessarily a vertex.
\end{uwg}

\section{Infinite systolic groups are not torsion}\label{r:torsion}
Here we turn to studying systolic groups, i.e.\ groups acting geometrically on systolic complexes. The aim of this section is to prove that an infinite systolic group contains an element of infinite order. This result is true for hyperbolic groups \cite[Proposition~III.$\Gamma$.2.22]{bridson1999metric} and for $\mathrm{CAT}(0)$ groups \cite[Theorem~11]{swensoncut}. As pointed out by the Referee, it is also true for systolic groups. This follows from biautomaticity of systolic groups shown in \cite[Theorem 13.1]{januszsimplicial} and the fact that all infinite biautomatic groups are not torsion (see \cite{gilman}).


However, our proof is direct and more elementary. It is based on the approach used in $\mathrm{CAT}(0)$ case, but instead of the usual geodesics we use the directed geodesics discussed in the previous section. Before proving Theorem~\ref{nottorsion} we show the existence of infinite directed geodesics in infinite systolic complexes.

\begin{lem}
\label{twr:istniejegeodesic}
Let $X$ be a locally finite, infinite systolic complex. Then $X$ contains an infinite directed geodesic.
\end{lem}
\begin{proof}Pick a vertex $v_0 \in X$. We will construct the required geodesic by defining a sequence of projections onto $v_0$. Since $X$ is infinite and locally finite, it contains an infinite $X^{(1)}$--geodesic. We can take this geodesic to be issuing from $v_0$, so in particular for every $n$, the sphere $S_n(v_0, X)$ is non-empty.

Now consider the set of simplices in $S_1(v_0, X)$, which are projections in the direction of $v_0$ of some simplices of $S_2(v_0, X)$. Call this set $P(1,2)$. This set is non-empty, because $S_2(v_0, X)$ is non-empty. Similarly, for any $n$ we define the set $P(1,n)$ to be the set of all simplices of $S_1(v_0, X)$ which come from iterating a projection of some simplex of $S_n(v_0, X)$, namely \[P(1,n) = \{ (\pi_{v_0})^{(n-1)}(\sigma) \ | \  \sigma \subset S_n(v_0, X)\}.\] 

Again, any such a set is non-empty. Moreover, directly from the definition of these sets we have $P(1,n) \subset P(1,n-1)$. So we have a descending family of sets $\{P(1,n)\}_{n>1} \subset S_1(v_0, X)$. Since $X$ is locally finite the sphere $S_1(v_0, X)$ is finite, so for $N$ sufficiently large this family stabilizes, i.e.\ for any $n\geq N$ we have $P(1,n)$=$P(1,N)$, and obviously $P(1,N)$ is non-empty. Pick any simplex of $P(1,N)$, and call it $\sigma_1$.

Now define the set $P(m,n)$ to be a subset of $S_m(v_0,X)$ consisting of simplices which project onto $\sigma_{m-1}$, and come from the iterated projection of some simplex of $S_n(v_0,X)$ $(n>m)$. Again, for any $m$, the family $\{P(m,n)\}_{n>m}$ has the same properties as for $m=1$, i.e.\ is non-empty, descending and stabilizes for $n$ sufficiently large. Hence, we can define $\sigma_m$ to be any simplex of $P(m,N_m)$ where $N_m$ is taken such that $P(m,N_m)$ is stable, i.e. for all $n \geq N_m$ we have $ P(m,n)=P(m,N_m)$.

This procedure gives us an infinite sequence of simplices $(v_0, \sigma_1,\ldots)$, which is a directed geodesic by Lemma~\ref{lem:projectionray}, because by construction $\sigma_i$ is the projection of $\sigma_{i+1}$ in the direction of $v_0$.
\end{proof}


\begin{defi} Let $G$ be a finitely generated group which acts on a topological space $X$. We say that the action is:

\begin{itemize}
\item \emph{proper} if for every compact subset $K \subset X$ the set $\{g \in G  \ | \ gK \cap K \neq \emptyset \} $ is finite,

\item  \emph{cocompact} if there exists a compact subset $K \subset X$ such that $X=G  K$.

\end{itemize}
\end{defi}

\begin{defi}A finitely generated group $G$ is called \emph{systolic} if it acts properly and cocompactly by simplicial automorphisms on a systolic complex $X$.
\end{defi}

\begin{twr}
\label{nottorsion}
Let $G$ be an infinite systolic group. Then $G$ contains an element of infinite order.
\end{twr}

\begin{proof}We follow the idea of Swenson's proof of an analogous result for $\mathrm{CAT}(0)$ groups \cite[Theorem~11]{swensoncut}. Let $X$ denote the systolic complex on which $G$ acts. Note that because $G$ is infinite and the action is proper and cocompact, the complex $X$ has to be locally finite and infinite. Thus, by Lemma~\ref{twr:istniejegeodesic} we can pick a vertex $v_0 \in X$ and an infinite directed geodesic $(v_0, \sigma_1,\ldots)$. We denote this geodesic by $\gamma$. We will use $\gamma$ to construct an element of infinite order.

Pick any sequence $(\sigma_k)$ of disjoint simplices of $\gamma$. We will be modifying this sequence by passing to a subsequence many times, which we will view as restricting the indices to the set $I$ for some $I \subset \mathbb{N}$. First note that since the action is cocompact, any simplex of $(\sigma_k)$ is in a translate of a compact set $K$. The set $K$ contains finitely many simplices, hence there are only finitely many possible values of $\mat{dim}\sigma_k$. Thus one particular dimension $d_0$ appears in $(\sigma_k)$ infinitely many times, and we take a subsequence $(\sigma_k)_{k \in I}$ such that $\mat{dim} \sigma_k =d_0$. Again by cocompactness, for every $\sigma_k$ there exists $g_k \in G$ such that $g_k(\sigma_k) \subset K$. Since $K$ is compact, there is a subsequence of $g_k(\sigma_k)$ which converges to some simplex $\tau_0 \subset K$. Because $G$ acts by simplicial automorphisms, for $k$ sufficiently large we have $g_k(\sigma_k)=\tau_0$.

Next, for every $\sigma_k$ we look at simplices $\sigma_{k-1}$ and $\sigma_{k+1}$. Again we have only finitely many possibilities for each of their dimensions, hence we can pass to a subsequence $(\sigma_k)_{k \in I}$ such that for every $k$ we have $\mat{dim}\sigma_{k-1}=d_{-1}$ and $\mat{dim}\sigma_{k+1}=d_{1}$. We have $\{\sigma_{k-1}, \sigma_{k+1} \} \subset B_1(\sigma_k,X)$, so also $\{g_k(\sigma_{k-1}),$ $ g_k(\sigma_{k+1}) \} \subset B_1(\tau_0,X)$. The ball $B_1(\tau_0,X)$ is compact, thus passing to a subsequence we get that $g_k(\sigma_{k-1}) \rightarrow \tau_{-1}$ and $g_k(\sigma_{k+1}) \rightarrow \tau_{1}$ for some simplices $\tau_{-1}$ and $\tau_1$ contained in $B_1(\tau_0,X)$. So taking $k$ sufficiently large we obtain $g_k(\sigma_{k-1}, \sigma_{k}, \sigma_{k+1})= (\tau_{-1}, \tau_0, \tau_{1})$.

Having $(\sigma_k)$ modified that much, we are ready to construct the desired element. Take $l>k$ such that $g_k$ and $g_l$ both satisfy the above condition, i.e.\ $g_k(\sigma_{k-1}, \sigma_{k}, \sigma_{k+1})=g_l(\sigma_{l-1}, \sigma_{l}, \sigma_{l+1})=(\tau_{-1}, \tau_{0}, \tau_{1}),$ and consider the element $h=g_l^{-1}g_k$. We claim that $h$ is of infinite order.

Indeed, the sequence $(\sigma_k,\ldots,\sigma_l) \subset \gamma$ is a directed geodesic, and so is the sequence $(h(\sigma_k),\ldots,h(\sigma_l))$, because $h$ is an $X^{(1)}$--isometry. Now consider the concatenation of these two. By construction of $h$ we have $(h(\sigma_{k-1}),h(\sigma_{k}),$ $h(\sigma_{k+1}))= (\sigma_{l-1},\sigma_{l},\sigma_{l+1})$, hence $(\sigma_{k-1},\ldots,\sigma_{l+1})$ and $(h(\sigma_{k-1}),\ldots,h(\sigma_{l+1}))$ intersect at three simplices. This is enough for $(\sigma_{k},\ldots,\sigma_{l}=h(\sigma_{k}),\ldots,h(\sigma_{l}))$ to be a directed geodesic, because the condition defining directed geodesic involves checking triples of consecutive simplices, and every triple in the concatenation is either contained in $(\sigma_{k-1},\ldots, \sigma_{l+1})$ or $(h(\sigma_{k-1}),\dots, h(\sigma_{l+1}))$. Iterating this procedure shows that for any $n>0$ the concatenation of $(h^m(\sigma_{k}),\ldots,h^m(\sigma_{l}))$ for $m=1$ to $n$ is a directed geodesic.

To finish the argument, assume that $h^n=e$ for some $n>0$. In particular this means $h^n(\sigma_k)=\sigma_k$, which in a view of Theorem~\ref{twr:directedisnormal} contradicts the fact that $(\sigma_k,\ldots,h^n(\sigma_k))$ is a directed geodesic.
\end{proof}

\end{document}